\documentclass[smallextended]{svjour3}       
\smartqed  

\usepackage{amsmath,amscd,amsfonts,amssymb,enumerate,amsthm} 
\usepackage{graphicx} 
\usepackage{listings}
\usepackage{rotating}
\usepackage{algorithm,algorithmic}
\usepackage{cite}
\usepackage{mathtools}
\usepackage{color}
\usepackage{url,hyperref}
\usepackage{breakurl}
\usepackage{verbatim}
\usepackage{comment}
\usepackage{pstricks}
\usepackage{pst-node}
\usepackage{pst-plot}
\usepackage{pst-poly}
\usepackage{caption,subcaption} 
\usepackage{enumerate}
\usepackage{pgfpages} 

\newcommand{\mh}[0]{\color{red}}
\newcommand{\hm}[0]{\color{blue}}
\newrgbcolor{darkgreen}{0 0.8 0}
\newcommand{\hg}[0]{\color{darkgreen}}
\DeclareMathOperator{\diag}{diag}	
\DeclareMathOperator{\rank}{rank}	
\DeclareMathOperator{\vect}{vec}	
\DeclareMathOperator{\sgn}{sgn}	
\DeclareMathOperator{\conv}{conv}	

\DeclarePairedDelimiter\parentheses{\lparen}{\rparen}	
\DeclarePairedDelimiter\braces{\lbrace}{\rbrace}	
\def\st{\mathrm{s.t.}}
\newcommand{\R}[0]{{\mathbb{R}}}    
\def\Mid#1{{#1}^c}		
\def\Rad#1{{#1}^\Delta}		

\newcommand{\omace}[1]{\mbox{$\overline{#1}$}}	
\newcommand{\umace}[1]{\mbox{$\underline{#1}$}}  
\newcommand{\ovr}[1]{\mbox{$\overline{#1}$}}	
\newcommand{\uvr}[1]{\mbox{$\underline{#1}$}}	
\newcommand{\onum}[1]{\overline{{#1}}} 	
\newcommand{\unum}[1]{\underline{{#1}}} 	
\newcommand{\mmid}[0]{:}		%

\newtheorem{assumption}[theorem]{Assumption}
\begin{document}

\title{An Overview of Absolute Value Equations: From Theory to Solution Methods and Challenges}


\titlerunning{An Overview on Absolute Value Equations}        

\author{	Milan Hlad\'{\i}k \and 
           Hossein Moosaei$^*$ \and
	Fakhrodin Hashemi \and\\
	 Saeed Ketabchi  \and
		Panos M. Pardalos 	
					}
	


\institute{$*$ Corresponding Author \at
	\and
 	 Milan Hlad\'{\i}k  \at
	Department of Applied Mathematics, Faculty  of  Mathematics  and  Physics, Charles University, Prague, Czech Republic \\
	\email{hladik@kam.mff.cuni.cz}   
 	\and
	Hossein Moosaei \at
	Department of Informatics, Faculty of Science, Jan Evangelista Purkyně University, \'{U}st\'{i} nad Labem, Czech Republic\\
	\email{hmoosaei@gmail.com,  hossein.moosaei@ujep.cz }
	\and
	Fakhrodin Hashemi\at
	Department of Applied Mathematics, Faculty of Mathematical Sciences,  University of Guilan, Rasht, Iran \\
	\email{fhashemi@phd.guilan.ac.ir}
	\and 
		Saeed Ketabchi \at
	Department of Applied Mathematics, Faculty of Mathematical Sciences,  University of Guilan, Rasht, Iran \\
	\email{sketabchi@guilan.ac.ir}
	\and 
		Panos M. Pardalos
  \at
	Center for Applied Optimization, Department of Industrial and Systems Engineering, University of Florida, Gainesville, 32611, USA \\
	\email{pardalos@ise.ufl.edu}
						}

\date{Received: date / Accepted: date}
\maketitle
\begin{abstract} 
This paper provides a thorough exploration of the absolute value equations $Ax-|x|=b$, a seemingly straightforward concept that has gained heightened attention in recent years. It is an NP-hard and nondifferentiable problem and equivalent with the standard linear complementarity problem. Offering a comprehensive review of existing literature, the study delves into theorems concerning the existence and nonexistence of solutions to the absolute value equations, along with numerical methods for effectively addressing this complex equation. Going beyond conventional approaches, the paper investigates strategies for obtaining solutions with minimal norms, techniques for correcting infeasible systems, and other pertinent topics. By pinpointing challenging issues and emphasizing open problems, this paper serves as a valuable guide for shaping the future research trajectory in this dynamic and multifaceted field.
\end{abstract}
%
%

\paragraph{Nomenclature.}
We use $e=(1,\dots,1)^{\top}$ for the vector of ones and $e_k$ for the $k$th canonical unit vector (with convenient dimensions). 
By $A^T$ we denote the Hermitian transpose, $I_n$ the identity matrix of size~$n$, and $\diag(s)$ stands for the diagonal matrix with entries $s_1,\dots,s_n$. Further, $\rho(\cdot)$ denotes the spectral radius, and  $\sigma_{\min}(\cdot)$ and $\sigma_{\max}(\cdot)$ the minimum and maximum singular values, respectively. By default, we use the Euclidean norm for vectors, and for matrices, we use the spectral norm, i.e., $\|A\|=\sigma_{\max}(A)$. 
The other matrix norms that we use are the induced matrix $p$-norm defined as
$$
\|A\|_p\coloneqq \max_{x:\|x\|_p=1}\|Ax\|_p.
$$
In particular, $\|A\|_\infty=\max_{i}\sum_{j}|a_{ij}|$. 
The $i$th row and $j$th column of a matrix $A$ are denoted  $A_{i*}$ and $A_{*j}$, respectively. 
The sign of a real $r$ is defined $\sgn(r)=1$ if $r>0$ and  $\sgn(r)=-1$ otherwise. 
 The absolute value and the sign function are applied entry-wise on vectors and matrices. Also matrix comparisons, like $A \leq B$ or $A < B$, are interpreted component-wise.
For a vector $x\in\R^n$, we use the shortcut $D(x)\coloneqq\diag(\sgn(x))$; this enables us to write $|x|=D(x)x$.  
An interval matrix\label{dfIntMat}
$[\underline{M},\overline{M}]=[\Mid{M}-\Rad{M},\Mid{M}+\Rad{M}]$ is a set of matrices 
$\{M\mmid \underline{M}\leq M \leq \overline{M}\}
=\{M\mmid |M-\Mid{M}|\leq \Rad{M}\}$.
The convex hull of a set $S$ is denoted $\conv{S}$.


\section{Introduction}

In recent decades, in the field of optimization and numerical analysis, the absolute value equations have been considered by many researchers. 
The AVE plays an essential role due to its theoretical aspect and applications. Systems of this type were addressed in the 1980s (see \cite{Roh1984b,Roh1984c,Roh1989}) and termed absolute value equation (AVE) by Mangasarian and Meyer in~\cite{mangasarian2006absolute}.\footnote{Some authors call them piecewise linear systems~\cite{BruCas2008}.} 

A system of absolute value equations is represented as 
\begin{align}\label{Eq2}
Ax-|x|=b,
\end{align}
where $A\in \R^{n\times n}$, $b \in \R^n$ are given and $x\in\R^n$ is the unknown. As a slight generalization of the AVE, a system of generalized absolute value equation (GAVE) has the form

\begin{align} \label{EqGAVE}
Ax-B|x|=b;
\end{align}
we discuss it in Section~\ref{ssGAVE}.

It is worth noting that some scholars consider an alternative expression  $Ax + |x| = b$, as the canonical form of the AVE and $Ax + B|x| = b$ for the GAVE. For the sake of presenting the results uniformly throughout this paper, we reformulate the AVE and GAVE as given in \eqref{Eq2} and \eqref{EqGAVE}, respectively.

Due to the presence of absolute values in these systems, various computational challenges arise for both AVE and GAVE. Specifically, the fundamental problem of verifying the solvability of the AVE is NP-hard~\cite{mangasarian2007absolute}.

This comprehensive overview strives to present fundamental and noteworthy findings pertaining to theories, applications and solution methods. Additionally, we tackle challenges related to discovering solutions with minimum norms and addressing infeasible systems, along with other pertinent issues. The paper delves into the intricacies of these aspects, offering a deeper understanding of the subject matter. Additionally, it highlights the challenges encountered in the process, shedding light on potential areas for further research and development. Within this context, the paper presents a historical review and prominent issues concerning the AVE and its solution methods. Importantly, it refrains from passing judgment on the methods' efficiency or correctness. The paper's objectives are to prevent parallel work, provide a roadmap for generating new algorithms and methods to solve the AVE problems and evaluate the current situation of existing algorithms and methods from both analytical and numerical perspectives.

\paragraph{Motivations of Absolute Value Equation.}
Systems with absolute values arise naturally in many areas. Below, we list a few particular problems yielding absolute value equations. Notice that absolute value systems were also introduced as an abstract domain in computer programming~\cite{ChenWei2023}.
 
\begin{enumerate}[(1)]
 \item 
 One of the most important motivations was the linear complementarity problem (LCP)~\cite{cottle2009linear}. Both problems are equivalent (see Section~\ref{RAVELCP}), and this equivalence brings new perspectives for both sides.

\item 
A source of the AVE is the theory of interval computations, where AVE appears in the characterization of certain solutions of interval systems of linear equations~\cite{Roh1989} or in the characterization of the regularity of interval matrices~\cite{Roh2009}.

\item 
Continuous piecewise linear function can be represented by various means. The representation by the GAVE~\cite{GriBer2015} is an alternative to the max-min representation \cite{Ovch2002,Scho2012} and to the so-called canonical representation by an explicit formula involving arithmetic operations and absolute values~\cite{ChuaDen1988,LinUnb1995}.

\item 
Many other (NP-hard) problems can easily be formulated as the AVE or GAVE. This includes, for example, the Set-Partitioning problem or the 0-1 knapsack feasibility problem; see Theorem~\ref{thmNpHardSP}. 

\item 
The AVE naturally appear when solving certain differential equations, in particular, free-surface hydrodynamics \cite{BruCas2008}, such as flows in porous media \cite{BruCas2009,CasZan2012}. Another problems involve the absolute value explicitly; consider the following boundary value problem~\cite{yong2015iteration,noor2018generalized}:
\begin{align*}
    & \frac{d^2u}{dt^2}- |u|=f(t),\ \ f(t) \in C[a,b].     \\
   & u(a)=\alpha,\ u(b)=\beta.
\end{align*}
By using the finite difference method, this problem leads to the AVE~\cite{yong2015iteration}.

\end{enumerate}

\paragraph{Roadmap.}
The theorems for the existence and nonexistence of solutions to the AVE are presented in Section~\ref{ThResults}. Some extensions for the GAVE are considered in Section~\ref{ssGAVE}. 
Algorithmic aspects are addressed in Section~\ref{NumerRes}, where different approaches to solving theAVE and GAVE are presented. Finding the minimum norm solution or a sparse solution of the AVE, optimal correction of an infeasible AVE, and a relation to interval analysis is discussed in Section~\ref{Otherapro}. Challenges and open problems for future works are proposed in Section~\ref{Challenges}. 


 \section{Theoretical Results for the AVE}\label{ThResults}

In this section, we review the theoretical results for the AVE. We also bring some new insights.

\subsection{The Solution Set}\label{SolutionSet}

 The solution set of the AVE is denoted by $$\Sigma=\{x\in{\R}^n\mmid Ax-| x|=b\}.$$
If it is nonempty, it can have finitely or infinitely many solutions. In the first case, it has at most $2^n$ solutions lying in mutually distinct orthants. For example, if $A=0$ and $b=e$, then the AVE reads $|x|=e$; it possesses $2^n$ solutions and $\Sigma=\{\pm1\}^n$. 
In fact, for every natural $n$ there is an AVE system having exactly $n$ solutions.\footnote{Personal communication with Ji\v{r}\'{\i} Sgall.}
 
 Despite the fact that $\Sigma$ is not convex in general, it can be created by a union of at most $2^n$ convex polyhedra. An orthant decomposition demonstrates this clearly: Let $s\in\{\pm1\}^n$ and consider the orthant determined by the sign vector~$s$, characterized by $\diag(s)x\geq0$. Therefore, the solution set lying in this orthant is a convex polyhedron described
 $$
 \Sigma\cap\{x\in{\R}^n\mmid \diag(s)x\geq0\}
 =\{x\in{\R}^nmmid (A-\diag(s))x=b,\ \diag(s)x\geq0\}.
 $$
 
 \begin{figure}[t]
\begin{subfigure}[b]{0.47\textwidth}
\centering
\psset{unit=4.975ex,arrowscale=1.5}\footnotesize
\begin{pspicture}(-4,-4.9)(4,3)
\psaxes[ticksize=2pt,labels=all,ticks=all, showorigin=false, Dx=1,Dy=1]{->}(0,0)(-3.2,-4.4)(3,2)
\psset{linecolor=blue}
\qdisk(1,0){2.3pt}
\qdisk(-1,-4){2.3pt}
\qdisk(-1,1.3333){2.3pt}
\end{pspicture}
\caption{
$\begin{pmatrix}0&0\\-1&-0.5\end{pmatrix}x-|x|=\begin{pmatrix}-1\\-1\end{pmatrix}$.}\label{figSolSet1}
\end{subfigure}
\hfill
\begin{subfigure}[b]{0.47\textwidth}
\centering
\psset{unit=4.975ex,arrowscale=1.5}\footnotesize
\begin{pspicture}(-2.7,-2.9)(6,4)
\psaxes[ticksize=2pt,labels=all,ticks=all, showorigin=false]{->}(0,0)(-2.2,-2.4)(5.2,3.2)
\psline[linewidth=1.5pt,linecolor=blue](3,0)(5.2,2.2)
\psset{linecolor=blue}
\qdisk(-1,-2){2.3pt}
\end{pspicture}
\caption{
$\begin{pmatrix}0&1\\-2&3\end{pmatrix}x-|x|=\begin{pmatrix}-3\\-6\end{pmatrix}$.}\label{figSolSet2}
\end{subfigure}
\caption{Two examples of AVEs.\label{figSolSet}}
\end{figure}
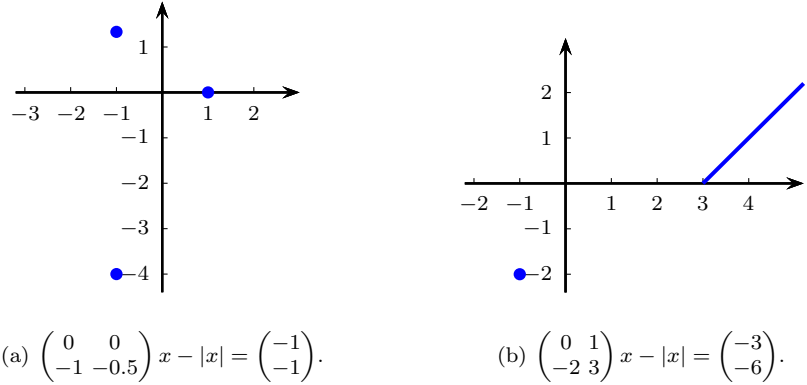


Figure~\ref{figSolSet} provides an illustration of two AVE systems. 
In the first case, the solution set consists of three points, and in the second case, it is formed by the union of a singleton and a ray. 

\subsection{Relationships Between the AVE and the LCP} \label{RAVELCP}

 The linear complementarity problem (LCP)~\cite{cottle2009linear} is a well-established problem in mathematical programming. 
There is a strong connection between the AVE and theLCP~\cite{mangasarian2006absolute,mangasarian2007absolute}. These problems are not only equivalent but also serve as a source for deriving basic properties of the AVE.

Recall that the LCP~\cite{cottle2009linear} is formulated as a feasibility problem with variables $w,z\in \R^n$
	\begin{align}\label{lcp}
	w=Qz+q,\ w^{\top}z=0,\ w,z\geq0.
		\end{align}
The LCP appears in many optimization problems, including quadratic programming, bimatrix games, and economic problems. 
Below, we present the equivalence between the AVE and the LCP \cite{mangasarian2006absolute,mangasarian2007absolute}.

\paragraph{Reduction AVE $\rightarrow$ LCP.} 
Suppose that $A-I_n$ is nonsingular (a reduction avoiding this assumption was proposed in~\cite{prokopyev2009equivalent}). The positive and negative parts of $x$ are denoted by $x^+$ and $x^-$, respectively. They can be expressed as follows:
$$
x=x^+-x^-,\ |x|=x^++x^-,\ (x^+)^{\top}x^-=0,\ x^+,x^-\geq0.
$$
Thus, the AVE takes the form
\begin{align}\label{aveAsLCP}
A(x^+-x^-)-x^+-x^-=b,\ (x^+)^{\top}x^-=0,\ x^+,x^-\geq0,
\end{align}
from which we obtain the LCP problem
$$
x^-=(A+I_n)^{-1}(A-I_n)x^+-(A+I_n)^{-1}b,\ (x^+)^{\top}x^-=0,\ x^+,x^-\geq0.
$$

\paragraph{Reduction LCP $\rightarrow$ AVE.}
We present the reduction from~\cite{mangasarian2007absolute}.
Suppose that $Q-I_n$ is nonsingular; this assumption does not affect generality since we can scale matrix $Q$ by any positive scalar. 
Now, the reduction is based on the substitution $w\equiv |x|-x$, $z\equiv |x|+x$. Hence, we can write the LCP as
$$
	|x|-x=Q|x|+Qx+q,
$$
which is equivalent to the AVE
$$
(I_n-Q)^{-1}(I_n+Q)x-|x|=(Q-I_n)^{-1}q.
$$
\paragraph{Computational Complexity.}
As we already mentioned, the AVE problem is intractable~\cite{mangasarian2007absolute}. We present the result with its illustrative proof.

\begin{theorem}\label{thmNpHardSP}
It is NP-hard to check if the AVE is solvable.
\end{theorem}

\begin{proof}
We use a reduction from the Set-Partitioning problem: 
\begin{center}
Given $a\in\mathbb{Z}^n$, exists $x\in\{\pm1\}^n: a^{\top} x=0$?
\end{center}
We write it as
$$
|x|=e,\ a^{\top} x=0.
$$
Equivalently, in the canonical form \eqref{Eq2} of size $n+2$, it reads
\begin{align*}
|x|=e,\ \ -a^{\top} x-|x_{n+1}|=0,\ \ a^{\top} x-|x_{n+2}|&=0.
\tag*{\qedhere}
\end{align*}
\end{proof}

Prokopyev~\cite{prokopyev2009equivalent} showed that it is NP-hard to check whether the AVE has at least two solutions (the AVE addressed there is easily transformed to the canonical form~\eqref{Eq2}). As a consequence, we have that, given one solution of the AVE, it is still hard to decide if there exists another solution. 
Another intractable problem, dealing with unique solvability for each right-hand side vector, is stated in Corollary~\ref{corAveUnSolNp}.

\subsection{Optimization Reformulations}\label{ssOptReform}

The solutions of AVEs can be determined as optimal solutions of certain nonlinear optimization reformulations \cite{mangasarian2006absolute,mangasarian2007absolute1,mangasarian2007absolute,mangasarian2015hybrid,zamani2021new}. In this subsection, we introduce some of them.

The first step is to rewrite the AVE system $Ax-|x|=b$ as inequalities
$$
Ax-|x|\geq b,\ \ Ax-|x|\leq b.
$$
The first inequality can be expressed as a simple linear inequality of double-size 
$$
(A+I_n)x\geq b,\ \ (A-I_n)x\geq b,
$$
the second part is the problematic one, and the role of the optimization reformulations is to achieve equality in $Ax-|x|\leq b$. The objective functions and the resulting optimization problems to this end are, e.g, 
\begin{align}
\label{minSum}
&\min\ e^{\top} (Ax-|x|-b)\ \ \st\ \ (A+I_n)x\geq b,\ (A-I_n)x\geq b,\\
\label{minAbsSq}
&\min\ (Ax-|x|-b)^{\top} (Ax-|x|-b)\ \ \st\ \ (A+I_n)x\geq b,\ (A-I_n)x\geq b,\\
\label{minPM}
&\min\ ((A + I_n)x - b)^{\top}((A- I_n)x - b)\ \ \st\ \ (A+I_n)x\geq b,\ (A-I_n)x\geq b,
\end{align}
see \cite{Hla2023a,mangasarian2006absolute,mangasarian2007absolute, mangasarian2007absolute1,mangasarian2015hybrid,zamani2021new}. The first objective function is concave and piecewise linear, the second is a piecewise convex quadratic, and the third is quadratic.
The AVE is solvable if and only if the optimization problem (any of the above) has an optimal value of~0. 

The objective function of \eqref{minPM} is strictly convex if and only if $A^\top A-I_n$ is positive definite, which is equivalent to the unique solvability sufficient condition that we state later in~\eqref{condRegSuff2}.

An alternative reformulation utilizes the transformation \eqref{aveAsLCP} and rewrites the AVE as
\begin{align*}
(A-I_n)x-(A+I_n)y= b,\ x,y\geq0,\ x^{\top} y=0.
\end{align*}
Now, keeping the linear constraints and minimizing the complementarity, we arrive at
\begin{align}
&\min\ x^{\top} y\ \ \st\ \ (A-I_n)x-(A+I_n)y= b,\ x,y\geq0.
\label{minBilin}
\end{align}





\paragraph{Linear Mixed 0–1 Reformulation.}  
Prokopyev~\cite{prokopyev2009equivalent} proposed an equivalent linear mixed 0–1 reformulation of the GAVE (and the AVE in particular). 
Assuming that $b\not=0$, the AVE can be reformulated as a linear mixed 0–1 problem 
\begin{align*} 
	&\max_{\alpha,x,y,z}\ \alpha \\ 
	&\st\ \ Ax - y= \alpha b,\;0 \leq y+x  \leq e-z,\ 0 \leq y-x \leq z,\ \alpha \geq 0, \ z \in \{0,1\}^n.
\end{align*}
Let $(\tilde\alpha,\tilde{x},\tilde{y},\tilde{z})$ be any feasible solution of this problem. If $\tilde{\alpha}>0$, then $x\coloneqq\tilde{x}/\tilde{\alpha}$ is a solution of AVE. Further, let $(\alpha^*,x^*,y^*,z^*)$ be an optimal solution of the 0-1 problem. If $\alpha^*=0$, then the AVE is infeasible. If $\alpha^*>0$, then $x\coloneqq x^*/\alpha^*$ is a minimum norm solution of the AVE (using the maximum norm).


The above formulation employs $2n+1$ continuous variables and $n$ binary variables. We can reduce the number of continuous variables if we have some initial bounds $x\in[\uvr{x},\ovr{x}]$ for the solutions (certain bounds will be discussed in Section~\ref{sssSolvty}). Then, we can rewrite the AVE equivalently as 
\begin{align*} 
Ax-y= b,\ \ -y\leq x\leq y,\ \ 
y\leq x - 2 \diag(\uvr{x}) z,\ \ 
y\leq -x + 2 \diag(\ovr{x}) (e-z),
\end{align*}
where $x,y\in\R^n$ and $z\in \{0,1\}^n$. Eliminating $y$, we obtain a mixed-integer linear system with  $n$ continuous variables and $n$ binary variables,
\begin{align*} 
(A-I_n)x &\geq b,\ \ &
(A-I_n)x&\leq b - 2 \diag(\uvr{x}) z,\\ 
(A+I_n)x &\geq b,\ \ &
(A+I_n)x&\leq b + 2 \diag(\ovr{x}) (e-z).
\end{align*}

\subsection{Existence of Solutions of the AVE}\label{sSolv}

Conditions for the solvability of the AVE can be obtained by transforming it into an LCP (as presented above) and utilizing conditions for the LCP. However, owing to the special structure of the AVE, alternative conditions were derived directly for the AVE.

\subsubsection{Solvability}\label{sssSolvty}

Solvability can be checked by adapting the optimization reformulations from Section~\ref{ssOptReform}. Consider the LP problem
\begin{subequations}\label{primal1}
\begin{align}
	&\min_{x,y}\ 0^{\top}x+0^{\top}y\\
	&\st\ \ Ax-y=b,\ x+y\geq 0,\ -x+y\geq 0,
\end{align}\end{subequations}
which results from the relaxation of the absolute value. Its dual reads
\begin{subequations}\label{dual1}
\begin{align}
	&\max_{u,v,w}\ b^{\top}u+0^{\top}v+0^{\top}w\\
	&\st\ \ A^{\top}u+v-w=0,\ -u+v+w= 0,\ v, w\geq 0.
\end{align}\end{subequations}
Based on the optimality conditions in LP and complementary slackness, 
Mangasarian~\cite{mangasarian2017sufficient} obtained the following sufficient condition for solvability.

\begin{proposition}
Let $(x, y)$ be an optimal solution of \eqref{primal1} and $(u, v, w)$ an optimal solution of~\eqref{dual1}. If $u>0$, then $x$ solves the AVE. Moreover, we have
\begin{align*}
x_{i}\geq 0 \mbox{ if } w_{i}>0,
\mbox{ \ and \ }
x_{i}\leq 0 \mbox{ if } v_{i}>0,\ \ i=1,\dots,n.
\end{align*}	
\end{proposition}	

Based on the Picard iterations (Subsection~\ref{PicardIter}), 
Mangasarian and Meyer~\cite{mangasarian2006absolute} obtained a condition for a nonnegative solution.

\begin{proposition}\label{prop7}
If $A\geq 0$, $\|A\| < 1$ and $b\leq 0$, then a nonnegative solution to the AVE exists.	
\end{proposition}	

The condition does not imply uniqueness; consider the AVE system $|x|=e$, for example. However, the condition does imply the uniqueness of a nonnegative solution (If there are more nonnegative solutions, then $A+I_n$ is singular, which contradicts the condition $\|A\| < 1$).

Notice that Radons and Tonelli-Cueto~\cite{RadTon2023} analysed (not necessarily unique) solvability of an AVE-type system by the aligned spectrum of the constraint matrix.

\paragraph{Bounds on the Solutions.}
From various aspects, it is important to have tight bounds on the solution set. Not only do they restrict the area where to seek for solutions, but also reduce the number of orthants to be processed. Later, we will see that these bounds also provide a test for unsolvability (Proposition~\ref{thmUnsolvHla}).
Hlad\'{\i}k~\cite{hladik2018bounds} proposed bounds in the form of a box that is symmetric around the origin. 

\begin{theorem}\label{propBoundEqAveA}
If $\rho(|A|)<1$, then each solution $x$ of the AVE satisfies
\begin{align}\label{boundPropBoundEqAveA}
|x|\leq -(I-|A|)^{-1} b.
\end{align}
\end{theorem}

Naturally, each solution $x$ must satisfy $|x|\leq Ax-b$. Together with the above bound, we obtain a polyhedral outer approximation of~$\Sigma$ in the form of a system of linear inequalities.

\begin{corollary}
If $\rho(|A|)<1$, then each solution $x$ of the AVE satisfies the system
\begin{align*}
(A+I_n)x\geq b,\ 
(A-I_n)x\geq b,\ 
 (I-|A|)^{-1} b\leq x\leq -(I-|A|)^{-1} b.
\end{align*}
\end{corollary}

\subsubsection{Unique Solvability} \label{UniqueSolve}

As in other mathematical problems, the uniqueness of a solution is an important issue. It is even more important in view of how many problems can be formulated as AVE, particularly the equivalence between the AVE and the LCP.
	
\paragraph{Characterization of Unique Solvability.}	
We present the characterization of Wu and Li~\cite{wu2018unique}  on unique solvability for an arbitrary right-hand side. Notice that the result can also be inferred from \cite{rohn2012rump,Rum2003}, and the sufficiency of the condition was already presented in \cite{rohn2004theorem,zhang2009global}. 
Recall that an interval matrix is regular if it contains only nonsingular matrices. 
	
\begin{theorem}\label{thmAveUnSol}
The AVE has a unique solution for each $b\in\R^n$ if and only if $[A-I_n,A+I_n]$ is regular.
\end{theorem}

A survey on regularity of interval matrices can be found in Rohn~\cite{Roh2009}. The conditions presented there enable us to characterize the unique solvability of the AVE by other means. We mention some of the most distinctive ones.

\begin{corollary}
The following conditions are equivalent:
\begin{enumerate}[(i)]
\item 
The AVE has a unique solution for each $b\in\R^n$;
\item
$[A-I_n,A+I_n]$ is regular;
\item
the system $|Ax|\leq|x|$ has only the trivial solution $x=0$;
\item
for each $s\in\{\pm1\}^n$, the linear system
$$
(A-\diag(s))x^1-(A+\diag(s))x^2=s,\ x^1,x^2\geq0
$$
is solvable;
\item
$\det(A+\diag(s))$ is constantly positive or negative for each $s\in\{\pm1\}^n$;
\item
each matrix in the form 
$$A+\diag(z^{(i)}),\quad i=1,\ldots,n,$$
is nonsingular, where $z^{(i)}_i\in[-1,1]$ and $|z^{(i)}_j|=1$ for $j\not=i$;
\item
for each $s\in\{\pm1\}^n$ and each real eigenvalue $\lambda$ of $\diag(s)A$ we have $|\lambda|>1$.
\end{enumerate}
\end{corollary}

Other characterizing conditions of unique solvability were presented in~\cite{KumDee2023u}.

Recall that the AVE is equivalent to the LCP. In the LCP, the matrix class providing unique solvability for each right-hand side is called \emph{P-matrices};  these matrices are defined as having all positive principal minors. A direct link between P-matrices and interval matrices was established in, for example,~\cite{rohn2012rump}.
According to Coxson~\cite{Cox1994}, verifying the P-matrix property is a co-NP-hard problem. Thus, we immediately have:

\begin{corollary}\label{corAveUnSolNp}
It is co-NP-hard to check if the AVE has a unique solution for each $b\in\R^n$.
\end{corollary}

Even though checking the regularity of $[A-I_n,A+I_n]$ is intractable in general, there are some easily recognizable classes. For instance, inverse nonnegative matrices are addressed in Section~\ref{sInvNonneg} or symmetric matrices~\cite{Hla2023a}.

\begin{proposition}
Let $A\in\R^{n\times n}$ be symmetric. Then $[A-I_n,A+I_n]$ is regular if and only if both matrices $A-I_n$ and $A+I_n$ have the same signature.
\end{proposition}

\paragraph{Sufficient Conditions.}
Due to the intractability of testing unique solvability, it is worth considering sufficient conditions. 

\begin{theorem}
Either of the following two conditions is sufficient for the unique solvability of the AVE for any $b\in\R^n$:
\begin{align}
\label{condRegSuff1}
\rho(|A^{-1}|)&<1,\\
\label{condRegSuff2}
\sigma_{\min}(A)&>1.
\end{align}
\end{theorem}

Condition \eqref{condRegSuff1} was stated in Rohn et al.~\cite{rohn2014iterative}, and the unique solution of the AVE can be computed in polynomial time then~\cite{zamani2021new}.  Condition \eqref{condRegSuff2}, which is often equivalently stated as $\| A^{-1}\|<1$, was proposed by Mangasarian and Meyer~\cite{mangasarian2006absolute}, and also in this case the unique solution of the AVE is polynomially computable. Nevertheless, it is an open problem if the AVE is efficiently solvable in the general case with $[A-I_n,A+I_n]$ regular; cf.~\cite{GharGil2012}.

Conditions \eqref{condRegSuff1} and \eqref{condRegSuff2} are independent of each other; that is, no one implies the second one~\cite{rohn2014iterative}. 
Both conditions are consequences of the following more general sufficient condition~\cite{wu2020unique}.

\begin{theorem}\label{thmAveUnSolSuff}
The AVE has a unique solution for each $b\in\R^n$ if $\rho(A^{-1}D)<1$ for every $D\in[-I_n,I_n]$.
\end{theorem}

Notice that we cannot weaken condition \eqref{condRegSuff1} to the form $\rho(A^{-1})<1$. For instance, matrix
$$
A=\begin{pmatrix} -1&2 \\ -2&1 \end{pmatrix}
$$
satisfies condition $\rho(A^{-1})<1$, but the interval matrix $[A-I_n,A+I_n]$ is not regular---it contains the singular matrix
$$
\begin{pmatrix} -2&2 \\ -2&2 \end{pmatrix}.
$$

Wu and Guo~\cite{wu2016unique} proposed another sufficient condition. To state it, recall that matrix $M\in\R^{n\times n}$ is an \emph{M-matrix} if $M_{ij}\leq0$ for $i\not=j$ and $M^{-1}\geq0$. A matrix $M\in\R^{n\times n}$ is \emph{H-matrix} if the \emph{comparison matrix} $\langle M\rangle$ is an M-matrix, where the comparison matrix is defined as $\langle M\rangle_{ij}=|M|_{ij}$ if $i=j$ and $\langle M\rangle_{ij}=-|M|_{ij}$ if $i\not=j$. H-matrices thus extend the class of M-matrices.

\begin{proposition}\label{thmAveUnSolSuffHmat}
The AVE has a unique solution for each $b\in\R^n$ if $A-I_n$ is an H-matrix with a positive diagonal.
\end{proposition}	
	
The above sufficient conditions can still be expensive for large and sparse matrices. That is why Wu and Guo~\cite{wu2016unique} proposed a condition based on a certain strictly diagonally dominant matrix property.

\begin{proposition}
The AVE has a unique solution for each $b\in\R^n$ if
\begin{align*}
| a_{ii}| >1+\sum_{i\neq j}| a_{ij}| 
\qquad\forall i=1,2,\dots,n.
\end{align*}
\end{proposition}

\subsubsection{Nonnegative Solutions}\label{sInvNonneg}

Hlad\'{\i}k~\cite{Hla2023a} studied the existence of nonnegative solutions and their uniqueness. We have an efficient way to check it.

\begin{theorem}
The following conditions are equivalent:
\begin{enumerate}[(i)]
\item 
The AVE has a unique nonnegative solution for each $b\geq0$;
\item 
the AVE has a nonnegative solution for each $b\geq0$;
\item
$(A-I_n)^{-1}\geq0$.
\end{enumerate}
\end{theorem}


To achieve the situation that the nonnegative solution is the unique solution, the concept of inverse nonnegative matrices turns out to be useful. An interval matrix $[\umace{A},\omace{A}]$ is inverse nonnegative if $A^{-1}\geq0$ for each $A\in[\umace{A},\omace{A}]$. Kuttler~\cite{Kut1971} proved that inverse nonnegativity of $[\umace{A},\omace{A}]$ can be easily characterized as inverse nonnegativity of the lower and upper bound matrices. In our context, $[A-I_n,A+I_n]$ is inverse nonnegative if and only if $(A-I_n)^{-1}\geq0$ and $(A+I_n)^{-1}\geq0$.

\begin{proposition}\label{propSolvInvNonnegNonneg}
If $[A-I_n,A+I_n]$ is inverse nonnegative, then for each $b \geq 0$, the AVE has a unique solution, and this solution is nonnegative.
\end{proposition}

Moreover, the Newton method can effectively compute the unique solution (Section~\ref{ssNewtonMethod}), which yields the solution in at most $n$ iterations in this case.

The converse implication in Proposition~\ref{propSolvInvNonnegNonneg} can be proved under an assumption on the regularity of $[A-I_n,A+I_n]$; it is an open question if the assumption can be relaxed.

\begin{proposition}
Let $[A-I_n,A+I_n]$ be regular. If for each $b\geq0$ the AVE has a unique solution and this solution is nonnegative, then $[A-I_n,A+I_n]$ is inverse nonnegative.
\end{proposition}

\subsubsection{Exponentially Many $(2^n)$ Solutions} \label{TwoNSolve}
	

As we already observed, the AVE (\ref{Eq2}) can possess $2^n$ solutions. In this case, they have to lie in the interiors of mutually different orthants. 
We present two results from Hlad{\'\i}k~\cite{hladik2018bounds}.


\begin{proposition}\label{propEqAveAExp}
Let  $\rho(| A|)<1$ and $b<0$. Then there are $2^n$ solutions provided at least one of the conditions holds
\begin{enumerate}[(i)]
	\item
	$2| b| > G(I_n-| A|)^{-1}|b|$,
	\item
	$| b| > 2| A| ~ | b|$,
\end{enumerate}
where $G$ is the diagonal matrix with entries $1/(I_n-| A|)^{-1}_{11},\dots,1/(I_n-| A|)^{-1}_{nn}$.
\end{proposition}
	
It can be shown that the first condition is weaker so that $(i)\ \Rightarrow\ (ii)$. On the other hand, the second condition is more easy to check. Notice that another sufficient condition was presented by Mangasarian and Meyer~\cite{mangasarian2006absolute}, but the above is provably stronger.

\subsection{Nonexistence of Solutions for the AVE}\label{NoSolve}

In Section~\ref{RAVELCP}, we presented a reduction AVE $\rightarrow$ LCP. Based on it, we can see that when the AVE is solvable, then
\begin{align*}
(A-I_n)x^+ - (A+I_n)x^-=b,\ x^+,x^-\geq0
\end{align*}
is solvable. By the Farkas theorem of the alternative, the dual system is unsolvable. This brings us to the result from~\cite{mangasarian2006absolute}.

\begin{proposition}
If
\begin{align}\label{aveRelaxDual}
-y\leq A^{\top} y\leq y,\ b^{\top} y>0
\end{align}
is solvable, then the AVE is unsolvable.
\end{proposition}

Mangasarian and Meyer \cite{mangasarian2006absolute} also stated the following condition.

\begin{proposition}\label{thmAveUnsolvNormA}
Let $0\neq b\geq 0$ and $\| A\| < 1$. Then, the AVE has no solution.	\end{proposition}
 
Other two conditions for AVE's unsolvability were presented in~\cite{hladik2018bounds}. The first condition is based on the bounds in \eqref{boundPropBoundEqAveA}. Obviously, if the bounds yield an empty set, then there is no solution to the AVE.

\begin{proposition}\label{thmUnsolvHla}
The AVE is unsolvable if
\begin{align}\label{testEqAveANonex}
\rho(|A|)<1\ \mbox{ and }\ -(I_n-|A|)^{-1} b\mbox{ is not nonnegative.}
\end{align}
\end{proposition}

 \begin{proposition}
If $\rho(| A|)<1$ and 
\begin{equation}
2b_i > \frac{1}{(I-| A|)_{ii}^{-1}} (I_n-| A|)_{i*}^{-1}| b|,
\end{equation}
for some $i$,  then the AVE has no solution.
\end{proposition}

When $A\geq0$, then the above two conditions \eqref{aveRelaxDual} and  \eqref{testEqAveANonex} are almost equivalent; see~\cite{hladik2018bounds} for details.

\subsection{More on the Structure of the Solution Set}

Recall from Section~\ref{SolutionSet} that the solution set $\Sigma$ need not be convex or connected in general; however, it forms a convex polyhedron in each orthant. Thus, its structure can be complicated, and diverse situations may occur. However, some situations are forbidden. For example, it cannot happen that there are infinitely many solutions in each orthant~\cite{Hla2023a}. On the other hand, it may happen that $2^n-1$ orthants possess infinitely many solutions.

The case of finitely many solutions was analysed by Hlad\'{\i}k~\cite{Hla2023a}.

\begin{proposition}\label{propFinChar}
Let $A\in\R^{n\times n}$. The solution set of $Ax-|x|=b$ is finite for each $b\in\R^n$ if and only if $A+\diag(s)$ is nonsingular for each $s\in\{\pm1\}^n$.
\end{proposition}

The exponential number of matrices can hardly be substantially decreased; it was proved that the property is co-NP-hard to check, even when matrix $A$ has rank one. On the other hand, several sufficient conditions proposed by Hlad\'{\i}k~\cite{Hla2023a}. 

Even when the solution set is infinite, it still might be bounded. Boundedness is easy to characterize; however, the condition is NP-hard to check~\cite{Hla2023a}. Notice that, by convention, the empty set is considered as bounded.

\begin{proposition}
Let $A\in\R^{n\times n}$. The solution set of $Ax-|x|=b$ is bounded for each $b\in\R^n$ if and only if $Ax+|x|=0$ has only the trivial solution $x=0$.
\end{proposition}

Even though the solution set $\Sigma$ is not generally convex, there are special situations when convexity occurs (by convention, the empty set is convex). For instance, if $[A-I_n,A+I_n]$ is regular, then there is a unique solution, which is a convex set. More generally, if $\Sigma$ is located in one orthant, it is convex. Surprisingly, the converse implication is also true~\cite{Hla2023a}. 

\begin{proposition}\label{propConvOrth}
The solution set $\Sigma$ is convex if and only if it is located in one orthant only, i.e., there is $s\in\{\pm1\}^n$ such that $\diag(s)x\geq0$ for each $x\in\Sigma$.
\end{proposition}

The problem of checking convexity if $\Sigma$ is also intractable; it is NP-hard on the class with $b=0$ and $A$ having rank one.

Connectedness of the solution set is harder to characterize; no characterisation is known, only some sufficient conditions. Also, the computational complexity is not classified.

\subsection{Cayley-Like Transform}

\emph{The Cayley transform}~\cite{FalTsa2002} of a matrix $A\in\R^{n\times n}$, where $I_n+A$ is nonsingular, is defined to be 
\begin{align*}
    \mathcal{C}(A)=(I_n+A)^{-1}(I_n-A).
\end{align*}
The transformations between the AVE and the LCP from Section~\ref{RAVELCP} resemble the Cayley transform. Indeed, there is a strong connection. Denote
\begin{align*}
 \mathcal{AL}(A)=(A+I_n)^{-1}(A-I_n)
\end{align*}
the transformation of the matrix from reduction AVE$\to$LCP and similarly denote
\begin{align*}
 \mathcal{LA}(Q)=(I_n-Q)^{-1}(I_n+Q)
\end{align*}
the transformation of the matrix from reduction LCP$\to$AVE. The connection to the Cayley transform is given in the following observation.

\begin{proposition}
We have $\mathcal{AL}(A)=-\mathcal{C}(A)$ and $\mathcal{LA}(Q)=\mathcal{C}(-Q)$.
\end{proposition}

Fallat and Tsatsomeros~\cite{FalTsa2002} investigated the Cayley transform of positivity classes in the context of LCP. Properties of the transform $\mathcal{LA}$ and $\mathcal{AL}$ have not been studied so far, but they have the potential to derive new results on the solvability of the AVE and to show relations of the AVE and certain matrix classes of the LCP. 

To illustrate it, we present several results on the uniqueness of the AVE solutions. Proposition~\ref{propCayleyALMmat} is a corollary of Proposition~\ref{thmAveUnSolSuffHmat}, but we prove it in another way by utilizing the above transforms. Proposition~\ref{propCayleyALPmat} weakens the assumption and also the statement. Proposition~\ref{propCayleyALInvMmat} is new.

\begin{proposition}\label{propCayleyALMmat}
If $A-I_n$ is an M-matrix, then $\mathcal{AL}(A)$ is an M-matrix, and therefore the AVE is uniquely solvable for each $b\in\R^n$.
\end{proposition}

\begin{proof}
We have to show that $\mathcal{AL}(A)$ has a nonnegative inverse and is a Z-matrix, i.e., the off-diagonal entries are nonpositive. The former follows from the fact that $(A-I_n)^{-1}\geq0$ and thus
$$
\mathcal{AL}(A)^{-1}
 =(A-I_n)^{-1}(A+I_n)
 =I_n+2(A-I_n)^{-1}\geq0.
$$ 
To show the latter, write 
$$
\mathcal{AL}(A)
 =(A+I_n)^{-1}(A-I_n)
 =I_n-2(A+I_n)^{-1}.
$$ 
It is a Z-matrix since $A+I_n$ is an M-matrix and hence has a nonnegative inverse.
\end{proof}

Recall that a nonsingular matrix is an \emph{inverse M-matrix} if its inverse is an M-matrix~\cite{JohSmi2011}. Inverse M-matrices form a subset of P-matrices.

\begin{proposition}\label{propCayleyALPmat}
If $A-I_n$ in a P-matrix, then $\mathcal{AL}(A)$ is a P-matrix, and therefore the AVE is uniquely solvable for each $b\in\R^n$.
\end{proposition}

\begin{proof}
P-matrices are closed under inversion, so it is sufficient to show that $\mathcal{AL}(A)^{-1}$ is a P-matrix. This follows from the expression
$$
\mathcal{AL}(A)^{-1}
 =(A-I_n)^{-1}(A+I_n)
 =I_n+2(A-I_n)^{-1},
$$ 
since P-matrices are also closed under positive multiples and the addition of a nonnegative diagonal matrix.
\end{proof}

\begin{proposition}\label{propCayleyALInvMmat}
If $A-I_n$ is an inverse M-matrix, then $\mathcal{AL}(A)$ is an inverse M-matrix, and therefore the AVE is uniquely solvable for each $b\in\R^n$.
\end{proposition}

\begin{proof}
Since $(A-I_n)^{-1}$ is an M-matrix, the matrix
$$
\mathcal{AL}(A)^{-1}
 =(A-I_n)^{-1}(A+I_n)
 =I_n+2(A-I_n)^{-1}
$$ 
is an M-matrix, too. This means that $\mathcal{AL}(A)$ is an inverse M-matrix. Therefore, it is also a P-matrix, and the corresponding LCP has a unique solution for any right-hand side vector.
\end{proof}

From the above results, it might seem that unique solvability for each $b\in\R^n$ (i.e., regularity of $[A-I_n,A+I_n]$) can be achieved if $A-I_n$ is positive stable, that is, the real part of all eigenvalues of $A$ is greater than~1. Nevertheless, this is not true in general. As a counterexample, consider the matrix
$$
A=\begin{pmatrix}
-1&1.5\\-4&3.5
\end{pmatrix}.
$$
The eigenvalues of $A$ are $1.25\pm 0.9682\,i$, but the interval matrix $[A-I_n,A+I_n]$ is not regular since it contains the singular matrix
$$
A=\begin{pmatrix}
-1.5&1.5\\-4&4
\end{pmatrix}.
$$


\section{Generalized Absolute Value Equation}\label{ssGAVE}


The AVE can be extended by various means. First, we are concerned with the so-called generalized absolute value equation (GAVE), which have the form
\begin{equation}\label{Eq1}
 Ax - B|x| = b,
\end{equation}
where $ A,B\in{\R}^{m\times n}$ and $b\in{\R}^m $.
If $B$ is nonsingular, then we easily transform it to the AVE
\begin{align*}
B^{-1}Ax-|x| = B^{-1}b.
\end{align*}
However, to avoid matrix inversion and handle the singular case, we analyse the form of the GAVE directly.

Notice that there is another reduction of the GAVE to the AVE that proceeds as follows. Write GAVE \eqref{Eq1} as
$$
Ax-By=b,\ y=|x|,\ y=|y|.
$$
By inserting a dummy variable $z$, we arrive at an equivalent system
$$
Ax+By-|z|=b,\ y=|x|,\ y=|y|+z,
$$
which is easily converted to the canonical form of AVE \eqref{Eq2}
$$
\begin{pmatrix}0&I_n&0\\0&I_n&-I_n\\A&B&0\end{pmatrix}
\begin{pmatrix}x\\y\\z\end{pmatrix}
-\begin{vmatrix}x\\y\\z\end{vmatrix}
=\begin{pmatrix}0\\0\\b\end{pmatrix}.
$$
This kind of reduction increases the size of the system by the factor~3, but we avoid matrix inversion, and it works with no assumptions on~$B$. Consequently, the reduction shows that the GAVE is equivalent to the LCP; this was already proved by Prokopyev~\cite{prokopyev2009equivalent} by other means.

\paragraph{Unique Solvability.}
Conditions for solvability and unsolvability for the GAVE are more-or-less extensions of those for the AVE.
For instance, linear programming based unsolvability condition \eqref{aveRelaxDual} was extended to the GAVE by Prokopyev~\cite{prokopyev2009equivalent}.
The following generalizations of Theorems~\ref{thmAveUnSol} and~\ref{thmAveUnSolSuff} to the GAVE are credited to Wu and Shen~\cite{wu2020unique}; some additional conditions were proposed in~\cite{mezzadri2020solution}. An overview of necessary and sufficient conditions is provided in~\cite{KumDee2023u}.

\begin{theorem}\label{thmgg}
The GAVE \eqref{Eq1} has a unique solution for any $b\in{\R}^n$ if and only if matrix 	$A + BD$ is nonsingular for every $D\in[-I_n,I_n]$.
\end{theorem}

We cannot write the set of matrices $A + BD$, $D\in[-I_n,I_n]$, directly $A + B[-I_n,I_n]$ since this interval matrix evaluated by interval arithmetic yields
$A + B[-I_n,I_n]=[A-|B|,A+|B|]$, which overestimates the true set. However, this matrix is also useful; we will meet it again around Theorem~\ref{thmalternatives} and Algorithm~\ref{algSignancord}.

We can still formulate unique solvability by means of a standard interval matrix. By \cite{KumDee2023u}, it is equivalent to regularity of the interval matrix
\begin{align}\label{intMcDouble}
\begin{pmatrix}A&B\\{}[-I_n,I_n]&I_n\end{pmatrix}.
\end{align}
Standard characterization of regularity~\cite{Roh2009} yields condition~\eqref{itCorIntMcDouble2} of Theorem~\ref{thmGaveCharT}. Condition~\eqref{itCorIntMcDouble3} was derived directly in~\cite{KumDee2023u}.

\begin{theorem}\label{thmGaveCharT}
The following conditions are equivalent:
\begin{enumerate}[(i)]
\item\label{itCorIntMcDouble1}
Interval matrix \eqref{intMcDouble} is regular;
\item\label{itCorIntMcDouble2}
the system $Ax+By=0$, $|y|\leq|x|$ has only the trivial solution $x=y=0$;
\item\label{itCorIntMcDouble3}
the system
$
|A^\top x| \leq  |B^\top x| 
$
has only the trivial solution $x=0$.
\end{enumerate}
\end{theorem}

Due to the linearity of the determinant in each column, we obtain from Theorem~\ref{thmgg} a finite characterization of the unique solvability of the GAVE.

\begin{corollary}
The GAVE has a unique solution for any $b\in{\R}^n$ if and only if the determinant of $A + B\diag(s)$ is either positive or negative for each $s\in\{\pm1\}^n$.
\end{corollary}

The characterization of unique solvability implies a general approach for deriving sufficient conditions.

\begin{proposition}\label{thmggg}
If $A$ in \eqref{Eq1} is nonsingular and satisfies
\begin{equation}\label{solgave}
	\rho(A^{-1}B{D} ) < 1 
\end{equation}
for every $D\in[-I_n,I_n]$, then the GAVE has a unique solution for any $b \in \R^n$.
\end{proposition}

From the above statement, we immediately have the sufficient condition for unique solvability for an arbitrary right-had side~\cite{rohn2014iterative}
\begin{align}\label{sufCondGaveSolvRho}
	\rho(|A^{-1}B| ) < 1,
\end{align}
which is a generalization of~\eqref{condRegSuff1}. Further, since for each $D\in[-I_n,I_n]$
$$\rho(A^{-1}B {D} ) 
\leq  \| A^{-1}B {D} \|
\leq  \| A^{-1}B\| \cdot \|{D} \|
\leq  \| A^{-1}B\|,
$$
we obtain another sufficient condition~\cite{wu2020unique}
\begin{align}\label{sufCondGaveSolvNorm}
\| A^{-1}B\|<1,
\end{align}
extending the condition \eqref{condRegSuff2} of the AVE. 
An alternative condition avoiding any matrix inversion is~
\cite{wu2019note,mohamed2019unique}
\begin{align}\label{sufCondGaveSolvSigmaBA}
 \sigma_{\max}(B) < \sigma_{\min}(A).
\end{align}
Its drawback is that it is provably weaker since it implies
\begin{align*}
1 > \sigma_{\max}(B) \sigma_{\min}(A)^{-1}
  = \sigma_{\max}(B) \sigma_{\max}(A^{-1})
  \geq  \sigma_{\max}(A^{-1}B)
  = \| A^{-1}B\|.
\end{align*}

Even when the unique solvability condition is satisfied, there is still some effort to compute the solution. Explicit formulae for the solutions exist in very special cases only~\cite{rohn2014class}. 

\begin{theorem}
Let $B \geq 0$, $\rho(B)<1$, and denote $M=(I_n-B)^{-1}$. Suppose that there is $k$ such that $b_{i}\geq 0$ for each $i\neq k$. Then, the absolute value equation
$$x -B | x|= b$$
has a unique solution
$$
x = \max\braces*{Mb,\, Mb-\frac{2(Mb)_k}{2M_{kk}-1}(M-I_n)e_k}.
$$
\end{theorem}

Due to the intractability of the GAVE, it is worth
reducing the dimension of the problem. Rohn~\cite{rohn2014reduction} showed that GAVE~\eqref{Eq1} could be reduced to a smaller-sized GAVE under certain assumptions.

\paragraph{Theorem of the Alternatives.}
Recall that theorems of the alternatives play an important role in expressing the solvability and unsolvability of linear systems, and they are used to derive duality in linear programming and optimality conditions in nonlinear programming. 
Rohn~\cite{rohn2004theorem} proposed a certain kind of theorem of the alternatives for the GAVE. 

\begin{theorem}\label{thmalternatives}
Let $A, D\in {\R}^{n\times n}$, where $D\geq 0$. Then exactly one of the following alternatives holds:
\begin{enumerate}[(i)]
\item 
for each $B\in{\R}^{n\times n}$ with $| B|\leq D$ and for each $b\in{\R}^n$ the equation
\begin{equation*}
	Ax-B| x| =b
\end{equation*}
has a unique solution;
\item 
the inequality
\begin{equation*}
	| Ax|\leq D| x|,
\end{equation*}
has a nontrivial solution.
\end{enumerate}
\end{theorem}

Notice that the first condition is equivalent to regularity of interval matrix $[A-D,A+D]$. 
The second condition can be equivalently stated (it is not so obvious): There exist $\lambda \in[0,1]$ and $y\in\{\pm1\}^n$ such that the absolute value equation
\begin{equation*}
 Ax-\lambda\diag(y)D| x|=0
\end{equation*}
has a nontrivial solution.

Rohn~\cite{Roh2012f} also presented a type of theorem of the alternatives for the absolute value equation in the form $|Ax| - |B||x| = b$. Instead of the unique solvability, the theorem expresses the existence of $2^n$ solutions for each positive right-hand side vector~$b$. Consequently, he obtained several sufficient conditions for the existence of $2^n$ solutions.

\paragraph{Overdetermined GAVE.}
Overdetermined systems (i.e., when $m>n$) usually have no solution. Rohn~\cite{r2019overdetermined} observed that if  $\rank(A)=n$ and $\rho(| A^{\dagger}B|)<1$, where $A^{\dagger}$ stands for the Moore--Penrose pseudoinverse of~$A$, then the GAVE has at most one solution. Moreover, he also proposed iterations
\begin{align*}
x^{0}&=A^{\dagger}b,\\
x^{i+1}&=-A^{\dagger}B| x^{i}|+A^{\dagger}b,\quad i=0,1,\dots
\end{align*}
that converge to a unique point~$x^*$. If the GAVE is solvable, then $x^*$ is its solution.

\paragraph{Other Extensions.}
Other variants of generalized AVE were also discussed; see, e.g., \cite{KumDee2023,wu2021unique,ZhoWu2021}. In particular, Wu~\cite{wu2021unique} considered generalized AVE in the form
\begin{equation}\label{NGAVE}
 Ax + |Bx| = b,
\end{equation}
where $A,B \in {\R}^{n\times n}$, and $b\in{\R}^n$.
It is straightforward to reduce \eqref{NGAVE} to GAVE~\eqref{Eq1}
\begin{align*}
 Ax + |y| = b,\ \ Bx - y = 0,
\end{align*}
however, some conditions take a simpler form. For instance, unique solvability for any right-hand side vector is achieved if and only if $A+DB$ is nonsingular for every $D\in[-I_n,I_n]$. From this characterization also, the analogous sufficient conditions 
$$
\rho(|BA^{-1}|)<1\ \mbox{ and }\ \|BA^{-1}\|<1
$$
directly follow.

\subsection{Sylvester-Like Absolute Value Matrix Equations}

Hashemi~\cite{hashemi2021sufficient} introduced Sylvester-like absolute value matrix equations
\begin{equation}\label{SylvesterAVE}
AXB + C| X| D = E,
\end{equation}
where $A,C \in {\R}^{m\times n}$,  $B,D \in {\R}^{p\times q}$, and $E \in {\R}^{m\times q}$  are given and matrix $ X \in {\R}^{n\times p} $ is unknown.
 
Matrix equations can be transformed to standard equations by utilizing Kronecker product~$\otimes$ and vectorization $\vect(X)$ of matrix~$X$ (a transformation stacking the columns of $X$ into a single column vector). Thus, \eqref{SylvesterAVE} equivalently reads as
\begin{align*}
    (B^{\top} \otimes A + D^{\top} \otimes C) \vect(X) = \vect(E).
\end{align*}
Using this transformation, we can adapt the GAVE solvability conditions. Thus, Theorem~\ref{thmGaveCharT} gets the following form.

\begin{proposition}
The following conditions are equivalent:
\begin{enumerate}[(i)]
\item 
The system \eqref{SylvesterAVE} has a unique solution for each $F\in\R^{m\times q}$;
\item\label{itThmMaveCharac3}
the system $AXB+CYE=0$, $|Y|\leq|X|$ has only the trivial solution $X=Y=0$;
\item\label{itThmMaveCharac2}
the system $|A^\top XB^\top|\leq|C^\top XE^\top |$ has only the trivial solution $X=0$.
\end{enumerate}
\end{proposition}

However, not all conditions for the GAVE are easily extendable to the matrix equations. There are known some special cases only~\cite{KumDee2023u}. For instance, Theorem~\ref{thmgg} is extended as follows.

\begin{proposition}
The system $AX+B|X|=F$ has a unique solution for each $F\in\R^{n\times n}$ if and only if matrix $A + BD$ is nonsingular for every $D\in[-I_n,I_n]$.
\end{proposition}

Several sufficient conditions for unique solvability of \eqref{SylvesterAVE} w.r.t.\ any  $E \in {\R}^{m\times q}$ were proposed in \cite{hashemi2021sufficient,wang2021new}. They resemble those conditions for GAVE:
\begin{enumerate}
\item
$\rho(| A^{-1}C|)\rho(| DB^{-1}|) < 1$, 
which extends~\eqref{sufCondGaveSolvRho};
\item
$\|A^{-1}C\|~  \|DB^{-1}\|<1$, which is an analogy of~\eqref{sufCondGaveSolvNorm};
\item 
$\sigma_{\max}(C)\sigma_{\max}(D) < \sigma_{\min}(A)\sigma_{\min}(B)$, which is an analogy of~\eqref{sufCondGaveSolvSigmaBA}.
\end{enumerate}

Analogously, we adapt conditions for unsolvability. Thus, Proposition~\ref{thmAveUnsolvNormA} is extended to the matrix equations as follows~\cite{hashemi2021sufficient}.

\begin{proposition}\mbox{}
\begin{enumerate}[(i)]
\item
Let $C$ be a square nonsingular matrix, $0\neq C^{-1}E \geq 0,$ and
$$
\sigma_{\max}(A)\sigma_{\max}(B) < \sigma_{\min}(C).
$$
Then the Sylvester-like equation $AXB - C| X| = E$ has no solution.
\item
Let $D$ be a square nonsingular matrix, $0\neq ED^{-1} \geq 0,$ and
$$
\sigma_{\max}(A)\sigma_{\max}(B) < \sigma_{\min}(D).
$$	
Then Sylvester-like equation $AXB-| X| D=E$ has no solution.
\end{enumerate}
\end{proposition}

Another form of absolute value matrix equations was considered by Wu~\cite{wu2021unique},
\begin{equation}\label{SylvesterAVEnew}
AXB + | C X D| = E.
\end{equation}
Unique solvability for an arbitrary right-hand side matrix $E$ was shown under similar sufficient conditions:
\begin{enumerate}
\item
$\rho(| CA^{-1}|)\rho(| B^{-1}D|) < 1$;
\item
$\|CA^{-1}\|\cdot\|B^{-1}D\|<1$;
\item 
$\sigma_{\max}(C)\sigma_{\max}(D) < \sigma_{\min}(A)\sigma_{\min}(B)$.
\end{enumerate}

\subsection{Other Generalizations}

The absolute value is associated with the nonnegative orthant. Extensions of the GAVE that consider another convex cone exist, too. In particular, Hu et al.~\cite{hu2011generalized}  introduced the absolute value equations associated with the second-order cone. Therein, the absolute value function is not defined as the traditional sum of the positive and negative parts but as the sum of the projections onto the positive and negative second-order cones instead. Follow-up research in this direction involves \cite{huang2019convergent,miao2015generalized,miao2017smoothing,MioYao2022,nguyen2019unified}. The relation to general equilibrium problems was studied by Gajardo and Seeger~\cite{GajSee2014}.
Extensions to circular cones were explored by Miao and Yang~\cite{miao2015generalized}. 


Casulli and Zanolli \cite{CasZan2012} considered nonlinear systems, in which the absolute value is replaced with an integral of a specified function.

Tensor absolute value equations \cite{BeiKal2022, cui2022existence, CuiLia2022, DuZha2018, JiaLi2021} are a rather new area with few results and many unsolved questions; see also the paragraph in Section~\ref{Challenges} devoted to challenges and open problems.
\section{Numerical Methods for Solving Absolute Value Equation} \label{NumerRes}

A variety of iterative approaches have been proposed for solving the AVE (\ref{Eq2}) and GAVE (\ref{EqGAVE}). This section provides an overview of the fundamental methods commonly employed by algorithms to solve the AVE. For GAVE, the methods work in a similar fashion; we present the GAVE form only when it is meaningful.

In essence, finding a solution for the AVE (\ref{Eq2}), i.e., solving the system  $Ax - |x|-b  = 0$, is equivalent to finding the root of the nonsmooth function
\begin{align}\label{FunctionAVE}
f(x) = Ax - |x| - b.
\end{align}
Thus, our objective is to solve the equation:
\begin{align}\label{RootAVE}
f(x) = 0.
\end{align} 

\subsection{Optimization Approaches}

As we observed in Section~\ref{ssOptReform}, the AVE can be expressed by means of optimization modelling. In this section, we will look at several concave optimization-based strategies. Recall from \eqref{minSum} that $x^\star$ is a solution of the AVE if and only if it is an optimal solution of the following concave minimization problem and the optimal value is zero:
\begin{align}\label{Concave}
\min\ e^{\top}(Ax-b-|x|) \ \ \st\ \ x\in\mathcal{S},
\end{align}
where
\begin{align}\label{dfS}
\mathcal{S}
\coloneqq\{ x\mmid |x|\leq Ax-b\}
= \{x\mmid (A+I_n)x\geq b,\, (A-I_n)x\geq b\}
\end{align}
is the feasible set. 
Mangasarian~\cite{mangasarian2007absolute1}  suggested a successive linearization algorithm based on this concept; the pseudocode is given in Algorithm~\ref{sla}. Furthermore, he demonstrated that the method converges to a stationary point of problem~\eqref{Concave}. 
Let $g=e^{\top}(Ax-b-|x|)$ denote the concave objective function of problem~\eqref{Concave}, which is bounded below by zero on~$\mathcal{S}$. Now, we establish a significant result through the following theorem.

\begin{algorithm}
\caption{Successive linearization algorithm}\label{sla}
\begin{algorithmic}[1]
\STATE 
pick $x^0\in \R^n$ arbitrarily
\FOR{$k=0, 1, \dots $}
\STATE
compute $x^{k+1}$ a vertex optimal solution of the linear program
\abovedisplayskip=1ex\belowdisplayskip=-1ex
\begin{align*}
   \min\ (e^{\top}A-\sgn(x^{k})^{\top})x
   \ \ \st\ \ x \in\mathcal{S}
\end{align*}
\IF{$(e^{\top}A-\sgn(x^{{k}})^{\top})(x^{k+1}-x^k)=0$}
\RETURN{solution $x^{k+1}$}
\ENDIF
\ENDFOR
\end{algorithmic} 

\end{algorithm}

\begin{theorem}
Algorithm~\ref{sla} generates a finite sequence of feasible vertices $\{x^{1},x^{2},\dots,x^{\ell}\}$ with strictly
decreasing objective function values: $g(x^{1})>g(x^{2})>\dots>g(x^{\ell})$, such that $x^{\ell}$ 
satisfies the minimum principle necessary optimality condition:
\[
(e^{\top}A-\sgn(x^{\ell})^{\top})(x-x^{\ell})\geq 0,\qquad \forall 
x\in\mathcal{S}.
\]
\end{theorem}

Mangasarian~\cite{mangasarian2015hybrid} also proposed a hybrid method for solving the AVE. Rewrite the AVE as $y=|x|$ , $y=Ax-b$ and consider its relaxation
\begin{equation}\label{mnaZ}
\mathcal{Z}=\{(x,y)\in {\R^{2n}} \mmid y\geq x \geq -y,\, y\geq Ax-b\}.
\end{equation}
Then the AVE is feasible if and only if the following concave minimization problem has the optimal value zero:
\begin{equation}\label{Concave2}
\min_{x,y}\ e^{\top}(y-|x|)+e^{\top}(y-Ax+b)
 \ \ \st\ \ (x,y)\in\mathcal{Z}.
\end{equation}
The suggested hybrid strategy (Algorithm~\ref{hybridmethod}) consists of alternating between solving a linear programming linearization of the concave problem \eqref{Concave2} to generate a solution $(x, y)$ and using the signs of the solution $ x$ to linearize the AVE to a standard system of linear equations.

\begin{algorithm}
\caption{Hybrid method}\label{hybridmethod}
\begin{algorithmic}[1]
\STATE 
pick $x^0\in \R^n$ arbitrarily, 
choose $itmax$ (typically $itmax = 10$)
\FOR{$k=0, 1, \dots ,itmax$}
\STATE
compute $z^k$ a solution of $(A-D(x^k))z=b$
\STATE
linearize the concave minimization problem (\ref{Concave2}) around $z^k$ as follows:
\abovedisplayskip=1ex\belowdisplayskip=-0.1ex
\begin{align*}
\min_{x,y}\;-(e^{\top}A+\sgn(z^{k})^{\top})x+2e^{\top}y
\ \ \st\ \ (x,y)\in\mathcal{Z}
\end{align*}
\STATE
compute $(x^{k+1}, y^{k+1})$ the solution of this linear program
\IF{$x^{k+1}$ solves AVE to a given accuracy}
\RETURN{solution $x^{k+1}$}
\ENDIF
\ENDFOR
\end{algorithmic} 
\end{algorithm}


\begin{proposition}\label{converghybrid}
 If $(\sgn(z^k)-\sgn(x^k))^{\top}x^{k+1}=0$, Algorithm~\ref{hybridmethod} terminates in a finite number of iterations at a stationary point of \eqref{Concave2}.
\end{proposition}

The computational results illustrated the usefulness of Algorithm~\ref{hybridmethod} by solving $100\%$  of a sequence of $100$ randomly generated instances AVEs in $\R^{50}$ to $\R^{1000}$ to an accuracy of $10^{-8}$. In contrast, Algorithm~\ref{sla} solved $95\%$ of the problems to an accuracy of $10^{-6}$.

Zamani and Hlad\'{i}k \cite{zamani2021new} also developed a method based on the formulation \eqref{Concave}. First, by using below Example~\ref{ZHexampl1}, they demonstrated that Algorithm~\ref{sla} is not convergent to a solution of the AVE in general. Then, they introduced a new method, described in Algorithm~\ref{hladick}.  Therein, $V$ denotes the set of vertices of the set $\mathcal{S}$ from \eqref{dfS}, and $Adj_{\mathcal{S}}(x) \subseteq V$ denotes the adjacent vertices to the vertex $x \in V$.
Under mild conditions, Algorithm~\ref{hladick} converges to a solution of the AVE; the original formulation of the convergence is even more general than that presented below in Theorem~\ref{thmAlgZamHlaConv}.

\begin{example}\label{ZHexampl1}
Let 
\begin{equation*}
A=\begin{pmatrix}
3&1\\
6&5
\end{pmatrix},\ \ 
b=\begin{pmatrix}
3\\
10
\end{pmatrix},\ \ 
\bar{x}=\begin{pmatrix}
\frac{5}{3}\\
0
\end{pmatrix},\ \ 
x^{\star}=\begin{pmatrix}
1\\
1
\end{pmatrix}.
\end{equation*}
The unique solution $x^{\star}$ of AVE (\ref{Eq2}) is contrasted with the unique optimal solution $\bar{x}$ of the following linear program:
\begin{align}\nonumber
 \min\ (e^{\top}A-\sgn(\bar{x})^{\top})x \ \ \st\ \ (A+I_n)x\geq b,\, (A-I_n)x\geq b.
\end{align}
It means that Algorithm \ref{sla} is not necessarily convergent to a solution of the AVE. 
\end{example}

\begin{algorithm}[t]
	\caption{Concave minimization approach}\label{hladick}
\begin{algorithmic}[1]
\STATE 
pick $x^0\in \R^n$ arbitrarily
\FOR{$k=0,1,\dots$}
\STATE
compute $x^{k+1}$ a vertex solution of the linear program
\abovedisplayskip=1ex\belowdisplayskip=0ex
\begin{align*}
\min_x\ e^{\top}(A-D(x^{k}))x\ \ \st\ \  x\in\mathcal{S}
\end{align*}
\IF{ $e^{\top} (A-D(x^k))(x^{k+1}-x^k)=0$}
\IF{$f(x^k)=0$}
\RETURN{solution $x^{k}$}
\ELSE
\STATE
compute $Adj_{\mathcal{S}}(\bar{x})$ and take 
$x^{k+1}\in \arg\min_{x\in{Adj_{\mathcal{S}}}(\bar{x}) }f (x)$
\IF{$f(x^{k+1})\geq f(x^k)$}
\RETURN{approximate solution $x^{k}$}
\ENDIF
\ENDIF
\ENDIF
\ENDFOR
\end{algorithmic} 
\end{algorithm}

\begin{theorem}\label{thmAlgZamHlaConv}
 If $[A-I_n,A+I_n]$ is regular, then Algorithm~\ref{hladick} terminates in
a finite number of steps at a solution of the AVE. 
\end{theorem}

\paragraph{Other Optimization Approaches.}
Mangasarian \cite{mangasarian2012primal} proposed a bilinear optimization formulation of the AVE. It is based on a linear programming relaxation~\eqref{mnaZ}, its dual problem and achieves a sort of strict complementarity for some variables.

A similar idea is used in \cite{mangasarian2013absolute}; however, the linear programming relaxation is constructed in a different way. The variables are replaced with a difference between the positive and negative parts, which is relaxed.

From another perspective, it is evident that the AVE can be addressed by solving the optimization problem
\begin{align}\label{minAveQf}
    \min_{x\in\R^n}\ \|Ax-|x|-b\|^2.
\end{align}
However, the objective function is nonconvex and nonsmooth. 
Shahsavari and Ketabchi~\cite{shahsavari2021proximal} proposed two methods, the proximal difference-of-convex algorithm and the proximal subgradient method; by the numerical tests, the latter performed better.
To solve \eqref{minAveQf}, Moosaei et al.~\cite{moosaei2015some} utilized another approach; a generalized Newton method with a line-search based on the Armijo rule.

Noor et al.~\cite{noor2012iterative} approached solving the AVE by minimization of the piecewise quadratic function
$$
 x^{\top}Ax - | x|^{\top}x- 2b^{\top}x.
$$
Later, Iqbal et al.~\cite{iqbal2015levenberg} adapted the Levenberg–Marquardt method, which can be seen as a sort of combination of steepest descent and the Gauss–Newton method. 
Pham Dinh et al.~\cite{DinHo2017} approached the problem by means of DC (Difference of Convex functions) programming.

\subsection{Newton Methods}\label{ssNewtonMethod}

Recall from \eqref{FunctionAVE} that solving the AVE means finding a foot of a nonlinear function $f(x) = Ax - |x| - b$. To this end, we can employ Newton-like iterations. Since $f(x)$ is nonsmooth, we have to consider generalized Newton methods.

 
As a pioneering scholar, Mangasarian \cite{mangasarian2009generalized} developed a generalized Newton approach for solving the AVE.\footnote{Independently, Brugnano and Casulli \cite{BruCas2008} derived a Newton-type method for solving the AVE that results from solving a hydrodynamic model.} This approach is similar to the conventional Newton method, except it utilises a generalized Jacobian because of the nonsmoothness term of the AVE. A generalized Jacobian of the function (\ref{FunctionAVE}) is given by $\partial f(x)\coloneqq A-D(x)$, where $D(x)\coloneqq\diag(\sgn(x))$. Since $Ax-|x|=(A-D(x))x$, the Newton iteration is formulated as follows,
\begin{align}\label{Gn1}
x^{k+1}
 &=x^k-\big(A-D(x^k)\big)^{-1}(Ax^k-|x^k|-b)\nonumber\\
 &=(A-D(x^k))^{-1} b, \qquad k=0, 1, 2, \dots,
\end{align}
where $x^{0}\in\R^{n}$ is an arbitrary starting point. The method is displayed in Algorithm~\ref{algNewtonMethod}. From another point of view, 
\eqref{Gn1} is obtained by a linearization of the AVE around $x^k$, yielding a linear system $(A-D(x^k))x=b$.

\begin{algorithm}[t]
\caption{Generalized Newton method}\label{algNewtonMethod}

\begin{algorithmic}[1]
\STATE 
pick $x^0\in \R^n$ arbitrarily
\FOR{$k=0, 1, \dots $}
\STATE
$x^{k+1}\coloneqq (A-D(x^k))^{-1} b$
\IF{$x^{k+1}=x^{k}$}
\RETURN{solution $x^{k}$}
\ENDIF
\ENDFOR
\end{algorithmic} 
\end{algorithm}

Mangasarian \cite{mangasarian2009generalized} proved that the Newton iterations converge linearly to the unique solution of the AVE from any starting point~$x^0$ under the assumption that $\| A^{-1}\| <\frac{1}{4}$. The sufficient unique solvability condition \eqref{condRegSuff2} reads as $\| A^{-1}\| <1$, so it is a natural question whether the assumption can be weakened to this form. Indeed, 
Theorem~\ref{WeekerManga} below uses the more general condition. Further, Griewank et al.~\cite{GriBer2015} showed that if $\| A^{-1}\|_p <\frac{1}{3}$ for some $p$-norm, then the method converges globally in finitely many iterations. Radons and Rump~\cite{RadRum2022} showed that the Newton method terminates in at most $n+1$ steps provided at least one of the four conditions for the Signed Gaussian elimination (page~\pageref{condSgnGE}) holds true.

\paragraph{Smoothing Methods.} 
Since the AVE is a nonsmooth problem, it is tempting to approximate the nonsmooth absolute value function by a smooth approximate \cite{caccetta2011globally,cruz2016global,li2016onmodified,saheya2018numerical,wang2011interval,zhang2009global}. 

Zhang and Wei  \cite{zhang2009global} proposed a generalized Newton method, combining semismooth and smooth Newton steps. They used the smoothing function to construct a Newton iterate that provides a descent in the function value. The smoothing parameter is gradually reduced to zero using a strategic approach, leading to a solution of the AVE. The authors establish the global and finite convergence of the algorithm under the assumption that the interval matrix $[A-I_n,A+I_n]$ is regular, which ensures the uniqueness of the solution (see Theorem~\ref{thmAveUnSol}).

Caccetta et al.~\cite{caccetta2011globally} employed the smoothing Newton algorithm, as described in \cite{qi2002smoothing}, for solving the AVE.  The algorithm utilizes a smoothing function $G_{\epsilon}(x)=Ax-\sqrt{x^{2}+\epsilon^{2}} -b$, where $\epsilon$ is a smoothing parameter and $\sqrt{x^{2}+\epsilon^{2}}\coloneqq(\sqrt{x_1^2+\epsilon^2},\dots,\sqrt{x_n^2+\epsilon^2})^\top$. At each iteration of the algorithm, a linear equation involving a Jacobian matrix is solved to obtain a search direction. The algorithm also includes a line search technique to ensure the descent property. 

This method successfully solved 297 out of 300 instances with an accuracy of $10^{-6}$. Moreover, it exhibited an average iteration count of $5.67$ per instance.
The theoretical convergence is proved under a quite weak assumption (in comparison with other methods).

\begin{theorem}\label{WeekerManga}
Suppose that $\| A^{-1}\|<1$. Then, an infinite bounded sequence $\{x^k\}$ is generated by the smoothing Newton algorithm \cite{qi2002smoothing} and the whole sequence $\{x^k\}$ converges quadratically to the unique solution of the AVE. 
\end{theorem}
 


Tang and Zhou \cite{tang2019quadratically} proposed a descent
method for solving the GAVE. They employed a smoothing function and the descent direction 
$$
d = \partial h_{p}(w)^{- 1}
\parentheses*{- h_{p}(w) + \gamma\cdot\min\{ 1,g_{p}(w)\}e_1},
$$
where $\gamma \in ( 0,\sqrt{2} )$, $p>1$,  
$h_{p}(w) = h_{p}(\varepsilon,x) = \parentheses[\big]{\varepsilon,\, \parentheses[\big]{Ax + B\sqrt[p]{|\varepsilon|^{p} + |x|^{p}} - b}{}^\top}{}^\top$,
\(g_{p}(w) = \frac{1}{2}{\| h_{p}(w) \|}^{2}\ \), and
\(\varepsilon\) is a smoothing parameter. 
They proved that the method converges to the solution of the GAVE quadratically.

	 Since the generalized Newton method needs the exact solution of a linear system of equations, it can be computationally expensive and may not be justified. Therefore, Cruz et al.~\cite{cruz2016global} proposed an inexact semi-smooth Newton method, starting at $x^{0}\in \R^n$ and residual relative error tolerance $\theta\geq 0$, by
\begin{equation}\label{bound}
 \| (A-D(x^{k}))x^{k+1}-b\|\leq \theta\cdot f(x^{k}),\qquad k=0,1,2,\dots
 \end{equation}
	 Note that, in absence of errors, i.e., when $\theta=0$, the above iteration retrieves $(A-D(x^{k}))x^{k+1}=b$, yielding formula~\eqref{Gn1}.    
  Global Q-linear convergence \cite{jay2001note} was proved in the following theorem. 



 \begin{theorem}
Suppose that $\| A^{-1}\|<\frac{1}{3}$ and 
 $$
 0\leq \theta< \frac{1-3\| A^{-1}\| }{\| A^{-1}\|(\| A\|+3)}.
 $$
Then the inexact semi-smooth Newton sequence $\{x^k\}$ converges globally and Q-linearly to $x^{\ast} \in \R^n$, the unique solution of the AVE, as follows
$$ 
\| x^{k+1}-x^{\ast}\|
\leq
 \frac{\| A^{-1}\|}{1-\| A^{-1}\|} 
  \parentheses[\big]{\theta (\| A\| +3)+2}\| x^{k}-x^{\ast}\|,\quad  k=0,1,2,\dots
$$
 \end{theorem}

Eventually, notice that extensive numerical comparisons of smoothing functions were performed by Saheya et al.~\cite{saheya2018numerical}.

\paragraph{Some Other Variants of the Newton Method.} 
Motivated by Algorithm~\ref{algNewtonMethod},  a new version of the generalized Newton method was presented by Li in~\cite{li2016modified}. It is based on solving the equation
\begin{equation}\label{Eqli3}
f (x^k ) + ( \partial f(x^k) + I_n)(x^{k+1}-x^k ) = 0.
\end{equation}
Using $f(x^k)=(A-D(x^k))x^k-b$ and $\partial f(x^k)=A-D(x^k)$, vector $x^{k+1}$ is expressed as
 $$x^{k+1}=(A+I_n-D(x^k))^{-1}(x^k+b).$$
This modified generalized Newton method is globally and linearly convergent to the unique solution of the AVE provided 
  \[\|(A + I_n - D)^{-1} \|<\dfrac{1}{3}\]
 for any diagonal matrix $D$ with diagonal elements of $\pm1$ or~$0$.
 

To avoid the generalized Jacobian matrix from being singular and to further accelerate the convergence of the generalized Newton method, Cao et al.~\cite{cao2021relaxed} proposed a new relaxed method based on iterations
\begin{align*}
x^{k+1} \coloneqq 
 \parentheses*{A -\theta\cdot  D(x^k )}^{-1} (b+(1-\theta)|x^k|),
\end{align*}
where $\theta$ is a parameter.

Feng and Liu~\cite{feng2016improved} designed a generalized Newton method, which is based on the classical ideas of Ostrowski~\cite{ostrowski1973solution} and Traub~\cite{traub1982iterative} for solving nonlinear equations in the real space.

\subsection{Picard Iterations} \label{PicardIter}

Another type of an AVE algorithm is based on the Picard iterative approach. In this way, the AVE is written as follows: 
\begin{align*}
	Cx-F(x)=0,
\end{align*}
where $C\in\R^{n\times n}$ is nonsingular and $F\colon\R^n \rightarrow  \R^n$ is a nonlinear function. The iterative approach is then stated as follows: 
	\begin{align*}
	x^{k+1}=C^{-1}F(x^k), \quad k=0, 1, 2, \dots,
	\end{align*}
	where $x^0\in\R^n$ is an initial point. 
 
If we define $C\coloneqq A$ and subsequently set $F(x)=|x|+b$, then we obtain a basic iterative process outlined in Algorithm~\ref{rohn}. This methodology was originally proposed by Rohn et al.~\cite{rohn2014iterative}. They illustrated that when $\rho(|A^{-1}|)<1$, indicating the fulfillment of the sufficient unique solvability condition \eqref{condRegSuff1}, the approach exhibits linear convergence. Furthermore, they extended this strategy for solving the GAVE and demonstrated that, under specific conditions, the generated sequence linearly converges to the unique solution of the GAVE.

\begin{algorithm}[t]
\caption{Picard iterations}\label{rohn}
\begin{algorithmic}[1]
\STATE 
pick $x^0\in \R^n$ arbitrarily
\FOR{$k=0, 1, \dots $}
\STATE
$x^{k+1}\coloneqq A^{-1}\big(|x^k|+b\big)$
\IF{$x^{k+1}=x^{k}$}
\RETURN{solution $x^{k}$}
\ENDIF
\ENDFOR
\end{algorithmic} 
\end{algorithm}
	

\begin{theorem}\label{picardtheorem}
Let $A$ be nonsingular and $\rho(|A^{-1}|)<1$. Then Algorithm~\ref{rohn} initiated with $x^0\coloneqq A^{-1}b$ converges to the unique solution $x^{\star}$ of the AVE. 
\end{theorem}

Moreover, for each $k \geq 0$,
$$
\| x^{\star}-x^{k+1}\| \leq \sigma_{\max}(| A^{-1}|)\|x^{\star}-x^k\|.
$$
Since $\rho(|A^{-1}|)<1$ does not imply $\sigma_{\max}(| A^{-1}|)<1$, this can provide a weak convergence result. However, there is always a suitable vector norm such that
$$
\| x^{\star}-x^{k+1}\| \leq \varrho \|x^{\star}-x^k\|
$$
for certain positive $\varrho<1$. Basically, a similar idea showed that the solution of the AVE can be computed in polynomial time provided $\rho(|A^{-1}|)<1$; see~\cite{zamani2021new}.

\paragraph{Ill-conditioned Systems.}
Khojasteh Salkuyeh~\cite{salkuyeh2014picard} observed that the main problem with the Picard iteration method is that if the matrix A is ill-conditioned then in each iteration of the Picard method, an ill-conditioned linear system should be solved. For this reason,   he proposed Picard Hermitian and skew-Hermitian splitting (Picard-HSS) method as follows:
\begin{align*}
B(\alpha)x^{(k,\ell+1)}
 &=C(\alpha)x^{(k,\ell)}+| x^{(k)}| +b, 
 \quad \ell=0,1,\dots, \ell_{k}-1,\ k=0,1,\dots,
\end{align*}
where 
\begin{align*}
B(\alpha)=\frac{1}{2\alpha}(\alpha I_n+H)(\alpha I_n+S),\quad
C(\alpha)=\frac{1}{2\alpha}(\alpha I_n-H)(\alpha I_n-S)
\end{align*}
and
\begin{align*}
H=\frac{1}{2}(A+A^{H}),\quad  S=\frac{1}{2}(A-A^{H}).
\end{align*}
Herein, $\alpha$ is a positive constant, $\{\ell_k \}_{k=0}^{\infty}$ a prescribed sequence of positive integers, and $x^{(k,0)} = x^{(k)}$ is the starting point of the inner HSS iteration at $k$th outer Picard iteration. This leads to the inexact Picard iteration method, 
summarized in Algorithm~\ref{PicardHSSIM}.

\begin{algorithm}
\caption{The Picard–HSS iteration method}\label{PicardHSSIM}
 Let 
$H=\frac{1}{2}(A+A^{H})$ and $ S=\frac{1}{2}(A-A^{H})$ 
and a sequence $\{\ell_k \}_{k=0}^{\infty}$ of positive integers.
\begin{algorithmic}[1]
\STATE 
pick $x^{(0)}\in \R^n$ arbitrarily
\FOR{$k=0, 1, \dots $}
\STATE
put $x^{(k,0)}\coloneqq x^{(k)}$
\FOR{$\ell = 0, 1,\dots,\ell_{k}-1$}
\STATE
solve the following linear systems to obtain $x^{(k,\ell+1)}$:
\abovedisplayskip=1ex\belowdisplayskip=0ex
\begin{align*}
(\alpha I+H)x^{(k,\ell+\frac{1}{2})}
 &=(\alpha I-S)x^{(k,\ell)}+| x^{(k)}| +b,\\[-3pt]
(\alpha I+S)x^{(k,\ell+1)}
 &=(\alpha I-H)x^{(k,\ell+\frac{1}{2})}+|x^{(k)}|+b
\end{align*}
\ENDFOR
\STATE
 put $x^{(k+1)}\coloneqq x^{(k,\ell_{k})}$
\IF{ $x^{(k+1)}$ solves AVE to a given accuracy}
\RETURN{solution $x^{(k+1)}$}
\ENDIF
\ENDFOR
\end{algorithmic} 
\end{algorithm}

\begin{theorem}
Let $A \in\mathbb{C}^{n\times n}$ be positive definite 
and $\eta =||A^{-1}||_2 < 1$. Then for any initial guess $x^{(0)} \in\mathbb{C}^n$ and any sequence of positive integers $\ell_{k} , k = 0, 1, \dots$, the iteration sequence $\{x^{(k)}\}_{k=0}^{\infty}$ produced by the Picard–HSS iteration method converges to $x^{\ast}$ provided that $\ell \coloneqq {\liminf}_{ k\to\infty} \ell_k \geq N$, where $N$ is a natural number satisfying 
\[
\| T(\alpha)^s\|_{2}<\frac{1-\eta}{1+\eta}\quad \forall s\geq N,
\]
and $T(\alpha)\coloneqq(\alpha I_n+S)^{-1}(\alpha I_n -H)(\alpha I_n+H)^{-1}(\alpha I_n -S).$ 
\end{theorem}

The values $\ell_k$ in the inner Picard–HSS iteration steps are often problem-dependent and difficult to determine in actual computations. Moreover, the iteration vector cannot be updated in a timely manner. To this end, Zhu and Qi~\cite{zhu2018nonlinear} presented the nonlinear HSS-like iterative method to overcome the defect of the above mentioned method. 
 Zhang in \cite{zhang2015relaxed} proposed a relaxed nonlinear PHSS-like iterative method, which is more efficient than the Picard-HSS iterative method and is a generalization of the nonlinear HSS-like iteration method. 
His method avoids using an explicit inner iteration process, a disadvantage of the Picard–PHSS iterative method, while retaining its benefits. 
To avoid solving two linear subsystems at each iteration of the Picard-HSS algorithm, Miao et al.~\cite{miao2021picard} proposed a single-step inexact Picard iteration method. 
Another variant of HSS was presented by Li~\cite{li2016onmodified}, who analysed the convergence by means of a smoothing function.

\emph{Generalizations.}
As a generalization of the  Picard method, Wang et al.~\cite{wang2019modified} introduced an iteration method for solving the GAVE by involving a positive semi-definite matrix~$\Omega$. Formulating it for the AVE, we rewrite it as
$$
(A+\Omega)x=\Omega x+|x|+b
$$
and obtain the iterations
\begin{align}\label{MN}
x^{k+1}
&=(A+\Omega)^{-1}(\Omega x^{k}+| x^{k}|+b),\quad k=0,1,\dots
\end{align}
They established sufficient conditions for the linear convergence of this method.

\subsection{SOR-Like Methods}

The Successive Over-Relaxation (SOR) method is a popular iterative technique used to accelerate convergence when solving large-scale linear systems of equations. 
By updating the solution at each iteration using a weighted combination of the current and previous iterations, SOR aims to achieve faster convergence than certain other methods, such as Jacobi and Gauss-Seidel.

The SOR method has several advantages in the classical theory of linear equations, such as guaranteed convergence under certain conditions, especially for diagonally dominant or symmetric positive definite matrices. It can also be easily parallelized, making it suitable for large-scale problems.

For the linear equation $Ax = b$, the iterations of the SOR method can be expressed in a component-wise form as
\begin{align*}
x_i^{k+1} &= (1 - \omega) x_i^{k} + \frac{\omega}{a_{ii}}\left(b_i - \sum_{j=1}^{i-1} a_{ij}x_j^{k+1} - \sum_{j=i+1}^{n} a_{ij}x_j^{k}\right),
\quad  i = 1,\dots,n,
\end{align*}
where 
$\omega$ is the relaxation parameter.

Several SOR-like methods have been proposed to solve the AVE. These methods extend the SOR concept by incorporating suitable modifications to address the nonlinearity introduced by the absolute value function. They aim to exploit the structure and properties of the AVE to enhance convergence. By iteratively updating the solution and adjusting the relaxation parameter, they provide efficient and robust numerical schemes for solving the AVE. 

\begin{algorithm}
\caption{(The SOR-like iteration method)}\label{sorl}
Let $ \omega $ be a positive constant.

\begin{algorithmic}[1]
\STATE 
pick $x^0,y^0\in \R^n$ arbitrarily
\FOR{$k=0, 1, \dots $}
\STATE
put 
  $x^{k+1}\coloneqq (1- \omega)x^{k}+ \omega A^{-1}(y^{k} + b)$
\STATE
put 
 $y^{k+1}\coloneqq (1-\omega)y^{k} + \omega | x^{k+1} |$
\IF{$x^{k+1}$ solves AVE to a given accuracy}
\RETURN{solution $x^{k+1}$}
\ENDIF
\ENDFOR
\end{algorithmic}   
\end{algorithm}

 By reformulating the AVE as a two-by-two block nonlinear equation, Ke and Ma in~\cite{ke2017sor} proposed a SOR-like iteration method. In particular, rewrite the AVE as
 \begin{align*}
 Ax-y&=b,\\
-| x| +y&=0,
  \end{align*}
 that is,
 \begin{equation}\label{NEq}
 \begin{pmatrix}
 A&-I_n\\
 -D(x)&I_n
 \end{pmatrix}\;\begin{pmatrix}
 x\\
 y
 \end{pmatrix}=\begin{pmatrix}
 b\\
 0
 \end{pmatrix}.
 \end{equation}
 Based on this formulation, they suggested the SOR-like method for solving the AVE, as described in Algorithm~\ref{sorl}.

 Let $(x^{\ast}, y^{\ast})$ be the solution pair of the above  equation and $(x^{k}, y^{k})$ be generated by Algorithm~\ref{sorl}. Define the iteration errors as 
 \[
 e_{k}^{x} \coloneqq x^{\ast} - x^{k},\quad
 e_{k}^{y} \coloneqq y^{\ast} - y^{k}. 
 \]
 To state the convergence properties, define the vector norm
\[
\|(e_{k}^{x},e_{k}^{y})\|_e\coloneqq \sqrt{\| e^{x}\|^{2}+{\omega}^{-2}\| e^{y}\|^{2}}.
\]


 \begin{theorem}
 Denote $\nu\coloneqq \| A^{-1}\|$ and $\tau\coloneqq\frac{2}{3+\sqrt{5}}$. 
 If
 \[\nu<1\ \text{ and }\ 
 1-\tau<\omega<\min\braces*{1+\tau,\sqrt{\tau/\nu}},
 \]
 then
 \[ \|(e_{k+1}^{x},e_{k+1}^{y})\|_e<\|(e_{k}^{x},e_{k}^{y})\|_e
 \]
 holds for every $k = 0, 1, \dots$
 \end{theorem}

Guo et al.~\cite{guo2019sor} have further theoretically considered the SOR-like method. They found some new convergence conditions and a choice of the optimal relaxation parameter $\omega$ that we mention below. 


\begin{theorem}
Let $\| A^{-1}\|<1$. Suppose that all eigenvalues of $DA^{-1}$ are real for any diagonal matrix $D$ with diagonal elements of $\pm 1$ or~$0$. Then, the SOR-like method is convergent for $0<\omega\leq1.$
\end{theorem}


\paragraph{Other Approaches.}
Dong et al.~\cite{dong2020new} proposed a new SOR-like method based on the transformation to an equivalent system double-sized system. 
Another block reformulation was considered by Li and Wu in~\cite{li2020modified}. Ke~\cite{ke2020new} also reformulated the AVE as a double-sized system and proposed a combination of the Picard and SOR methods. Yu et al.~\cite{yu2020modified} modified this method to improve convergence properties. 
Zheng~\cite{zheng2020picard} considered a method combining the ideas of the Picard-HSS iterations and the SOR method.

\subsection{Matrix Splitting Methods}


Matrix splitting methods are iterative techniques used to solve systems of linear equations. They involve decomposing a given matrix into a sum of matrices, typically an easily invertible one and a matrix with desirable properties. By iteratively applying this splitting and updating an initial guess, these methods aim to converge to the solution of the linear system. In this subsection, we will review some of them that are applied to solving the AVE and GAVE.

The general scheme proposed by Zhou et al.~\cite{zhou2021newton} expresses  matrix $A$ as $A=M-N$ and the AVE as
\begin{align*}
  (M+\Omega)x=(N+\Omega)x +|x|+b,  
\end{align*}
where $\Omega\in\R^{n\times n} $ is an arbitrary matrix. 
The iterations then take the form as presented in Algorithm~\ref{NMS}.


\begin{algorithm}
\caption{Newton-based matrix splitting iterative method}\label{NMS}
Let $\Omega\in\R^{n\times n} $ be arbitrary, and $A=M-N$ a splitting such that $\Omega + M$ is nonsingular.
\begin{algorithmic}[1]
\STATE 
pick $x^0\in \R^n$ arbitrarily
\FOR{$k=0, 1, \dots $}
\STATE
put 
$x^{k+1}\coloneqq (M +\Omega)^{-1} ((N + \Omega)x^k + |x^k| + b)$
\IF{$x^{k+1}$ solves AVE to a given accuracy}
\RETURN{solution $x^{k+1}$}
\ENDIF
\ENDFOR
\end{algorithmic} 
\end{algorithm}

Notice that the case with  $M = A$ and $\Omega = N = 0$ reduces to the standard Picard method (Algorithm~\ref{rohn}). The case with  $M = A$ and $N = 0$ is the generalization of the Picard method discussed in~\eqref{MN}.

The convergence of Algorithm~\ref{NMS} is stated in the following theorem. 
Particularly for the case of standard Picard metod ($M = A$ and $\Omega = N = 0$), we obtain the condition $\rho(|A^{-1}|)<1$, which is the same as that in Theorem~\ref{picardtheorem}.

\begin{theorem}\label{thmnms}
Let $A\in {\R}^{n\times n}$ and $A = M - N$ be its splitting. Suppose that the matrix $M +\Omega$ is nonsingular, where $\Omega\in {\R}^{n\times n}$ is given matrix. 
If
$$\rho\parentheses*{|(M + \Omega)^{-1} (N + \Omega)|+|(M + \Omega) ^{-1} B|}<1,$$
then Algorithm~\ref{NMS} is convergent.
\end{theorem}


Another convergence result focuses on positive definite matrices. Note that  the case with $M = A$ and $\Omega = N = 0$ yields the condition $\|A^{-1}\|<1$.

\begin{theorem}
Let $A$ be positive definite and $A = M-N$ its splitting, where $M$ is positive definite. If matrix $\Omega$ has positive diagonal and
$$\| M^{-1}\|<\dfrac{1}{\| \Omega\|_{2}+\| \Omega+N\|_{2}+\| B\|_{2}},$$
then Algorithm~\ref{NMS} is convergent.
\end{theorem}

For linear systems, matrix splittings often split the matrix into lower and upper diagonal parts. Let $D, E$ and $F$ be the diagonal, strictly lower triangular and strictly upper triangular matrices constructed from $A$ such that $A=D-E-F$, then the AVE takes the form
 \begin{equation*}
 (D-E)x-| x| =Fx+b.
 \end{equation*}
This is the base for the adaptation of the Gauss-Seidel method to the case of the AVE, as was done by Edalatpour et al.~\cite{edalatpour2017generalization}. The iterations then read as
 \begin{equation}\label{Eqggs1}
 (D-E)x^{k+1}-| x^{k+1}| =Fx^{k}+b,\quad k=0,1,\dots
 \end{equation}
 They assumed that the diagonal entries of $A$ are greater than one. 
The method is described in Algorithm~\ref{algggss}; we present it in a slightly different fashion.

\begin{algorithm}[t]
\caption{Generalized Gauss-Seidel iteration method}\label{algggss}

\begin{algorithmic}[1]
\STATE 
pick $x^0\in \R^n$ arbitrarily
\FOR{$k=0, 1, \dots $}
\FOR{$i=0, 1, \dots,n$}
\STATE
put $s \coloneqq b_{i}-\sum_{j=1}^{i-1} a_{ij}x_{ j}^{k+1} - \sum_{j=i+1}^{n} a_{ij}x_{ j}^{k}$
\IF{$s \geq 0$}
\STATE
put $x_{i}^{k+1} \coloneqq \frac{s}{a_{ii}-1}$ 
\ELSE
\STATE
put $x_{i}^{k+1} \coloneqq \frac{s}{a_{ii}+1}$ 
\ENDIF
\ENDFOR
\IF{$x^{k+1}$ solves AVE to a given accuracy}
\RETURN{solution $x^{k+1}$}
\ENDIF
\ENDFOR
\end{algorithmic} 
\end{algorithm}

They also considered the Gauss-Seidel iteration method to solve a preconditioned AVE and suggested an efficient preconditioner to expedite the convergence rate of the method when matrix $A$ is a $Z$-matrix. For this reason, they transformed the AVE into a preconditioned system in the form of GAVE
\begin{equation}\label{pave}
P_{\beta}Ax-P_{\beta}|x| = P_{\beta}b,
\end{equation}
where $P_{\beta} = D + \beta F$. The Generalized Gauss–Seidel iteration method corresponding to the preconditioned system is described in Algorithm~\ref{pggss}.

\begin{algorithm}
\caption{Preconditioned generalized Gauss--Seidel}\label{pggss}
Let 
$ P_{\beta}A =\tilde{D}-\tilde{E}-\tilde{F}$
be the splitting to the diagonal, strictly lower and upper triangular matrices.

\begin{algorithmic}[1]
\STATE 
pick $x^0\in \R^n$ arbitrarily
\FOR{$k=0, 1, \dots $}
\STATE
solve the following system for $x^{k+1}$,
\abovedisplayskip=1.5ex\belowdisplayskip=-1ex
\begin{align*}
 D^{-1}(\tilde{D}-\tilde{E})x^{k+1}-| x^{k+1}|
 = \beta D^{-1}F| x^{k}| + D^{-1}\tilde{F} x^{k} + D^{-1}P_{\beta}b
\end{align*}
\IF{$x^{k+1}$ solves AVE to a given accuracy}
\RETURN{solution $x^{k+1}$}
\ENDIF
\ENDFOR
\end{algorithmic} 
\end{algorithm}

\paragraph{Other Approaches.}

He et al.~\cite{he2018two} proposed a method based on Hermitian and skew-Hermitian splitting, a sort of variant to the Picard–HSS iteration method (Algorithm~\ref{PicardHSSIM}).  
Shams et al.~\cite{shams2020iterative} proposed a novel approach called block splitting; it encompasses the Picard iterative method as a specific instance and closely relates with the SOR-like iterations. 
  Based on the shift splitting technique, Wu and Li in~\cite{wu2020special} proposed a special iteration method, which was obtained by reformulating the AVE equivalently as a two-by-two block nonlinear equation. 
  Based on the mixed-type splitting idea for solving a linear system of equations, a new algorithm for solving the AVE was presented by Fakharzadeh and Shams in~\cite{fakharzadeh2020efficient}. This algorithm utilizes two auxiliary matrices, which were limited to be nonnegative strictly lower triangular and nonnegative diagonal matrices. It was shown that by a suitable choice of the auxiliary matrices, the convergence rate of this algorithm outperforms standard variants of the Picard, generalized Newton, and SOR-like methods.
Picard conjugate gradient algorithm was proposed by Lv and Ma~\cite{lv2017picard}; convergence was proved for a symmetric positive definite matrix $A$ under the assumption of $\|A^{-1}\|_{2}<1$.

To avoid solving the linear systems at each iteration exactly, Yu et al.~\cite{yu2021class} developed an inexact version of the matrix splitting iteration method.
 
For matrix AVE in the form  $AX + B| X| = C$, Dehghan and Shirilord~\cite{dehghan2020matrix} developed a Picard-type matrix splitting iteration method. 

  \subsection{Other  Methods}

In this section, we present methods that approach to solving the AVE or GAVE in a different way. The fact that we mention them in this section does not mean that they are marginal. On the contrary, it turned out that they are very efficient in solving certain types of problems. 

\paragraph{The Sign Accord Algorithm.} 
This algorithm was originally introduced by Rohn~\cite{Roh1989} for computing the particular vertices of the convex hull of the solution set of interval linear equations. Later, Rohn~\cite{rohn2009algorithm} adapted it to solve a more general class of GAVE systems $Ax-B|x|=b$.

\begin{algorithm}[t]
\caption{The sign accord algorithm}\label{algSignancord}

\begin{algorithmic}[1]
\STATE 
$s\coloneqq \sgn(A^{-1}b)$  \label{algSignancord1}
\STATE 
$x\coloneqq (A-B\diag(s))^{-1}b$
\STATE 
$p\coloneqq 0\in\R^n$
\WHILE{$s_kx_k<0$ for some $k$}
\STATE
$k\coloneqq \min\{i\mmid s_ix_i<0\}$
\STATE
$p_k\coloneqq p_k+1$
\IF{$\log_2(p_k)>n-k$}
\RETURN{$[A-|B|,A+|B|]$ is not regular}
\ENDIF
\STATE
$s_k\coloneqq -s_k$
\STATE
$x\coloneqq (A-B\diag(s))^{-1}b$ \label{algSignancord9}
\ENDWHILE
\RETURN{solution $x$}
\end{algorithmic} 
\end{algorithm}

The scheme of the method is presented in Algorithm~\ref{algSignancord}. The underlying idea is that we try to find an agreement between a candidate solution $x$ and a sign vector~$s$; in that case, we have a solution of the Gthe AVE (recall that once we know the sign $s$ of the solution, then we are done; the solution is $x=(A-B\diag(s))^{-1}b$). For each iteration, Algorithm~\ref{algSignancord} switches the sign of one entry until we have a solution or exceed the number of iterations.

The algorithm terminates in a finite number of steps. It either finds a unique solution of the GAVE or states that interval matrix $[A-|B|,A+|B|]$ is not regular. By Theorem~\ref{thmalternatives}, the latter means that the GAVE system $Ax-B'|x|=b'$ has not a unique solution (in fact, has infinitely many solutions) for some $|B'|\leq |B|$ and $b'\in\R^n$.
 The case where $[A-|B|,A+|B|]$ is not regular is identified either by singularity of $A$ or $A-B\diag(s)$ or by  a high number of iterations ($\log_2(p_k)>n-k$). Later, Rohn~\cite{Roh2010c} improved the algorithm such that it always returns a singular matrix $S\in [A-|B|,A+|B|]$ in case $[A-|B|,A+|B|]$ is not regular. 

Let us give some practical details. 
The choice of sign vector $s$ in step~\ref{algSignancord1} is optional; the proposed choice is a heuristic supported by numerical experiments. To avoid solving a system of equations from scratch, the original Rohn's method utilizes the Sherman--Morrison formula for $x$ in step~\ref{algSignancord9}; indeed, the matrix change has rank one. 
 The algorithm demonstrates remarkable efficiency, averaging approximately $0.11\cdot n$ iterations per example.

\emph{Signed Gaussian elimination.}
Radons~\cite{Rad2016} proposed a modified Gaussian elimination to solve the GAVE in the form $x-B|x|=b$, that is, the GAVE with $A=I_n$. The key observation is that if $\|B\|_{\infty}<1$, then there is $i$ such that the sign of $b_i$ coincides with the sign of $x^*_i$, where $x^*$ is the unique solution of the system. Such $i$ can be hard to identify in general; however, there are certain classes of problems for which we can take any $i$ maximizing the value of $|b|_i$. There were found the following classes having this property (the last one was proved in \cite{RadRum2022}):
\begin{enumerate}\label{condSgnGE}
\item 
$\|B\|_{\infty}<\frac{1}{2}$,
\item 
$B$ is irreducible and $\|B\|_{\infty}\leq\frac{1}{2}$,
\item 
$B$ is strictly diagonally dominant and $\|B\|_{\infty}\leq\frac{2}{3}$,
\item 
$|B|$ is symmetric and tridiagonal, $n\geq2$, and $\|B\|_{\infty}<1$.
\end{enumerate}
Once we identify such index $i$, we can replace $|x|_i$ with $x_i$ and perform one step of full-step Gaussian elimination. The above conditions also remain valid for the reduced matrix after the elimination step. Therefore, one can proceed further and complete the Gaussian elimination. The computational cost is cubic, as for the classical Gaussian elimination. The cost of a tridiagonal matrix $B$ corresponds to sorting $n$ ﬂoating point numbers.

\emph{Metaheuristics.}
Since the AVE is generally an intractable problem, we cannot hope for an exact and efficient algorithm to solve the AVE in higher dimensions. That is what motivates the study and use of various heuristics and metaheuristics.

Mansoori et al.~\cite{mansoori2017efficient} designed an efficient neural network model to provide an analytical solution for the AVE. The model is based on the equivalence of the AVE and the linear complementarity problem. They proved the stability of the model using the Lyapunov stability theory and demonstrated global convergence. 

For the AVE with possibly multiple solutions, Moosaei et al.~\cite{moosaei2015minimum} proposed and implemented a simulated annealing algorithm.

\emph{Transformation to other problems.}
We saw in Section~\ref{RAVELCP} that the AVE is equivalent to the linear complementarity problem (LCP), so in principle, we can solve the AVE by using this transformation employing any solver for the LCP.

Another possible reduction is to mixed-integer linear programming, for which efficient solvers now exist. In Section~\ref{ssOptReform}, we presented some reformulations.

\emph{All solutions.}
Rohn~\cite{rohn2012algorithm} proposed an algorithm to generate all (finitely many) solutions of the GAVE system $Ax-B|x|=b$. It succeeds if and only if matrix $A-B\diag(s)$ is nonsingular for each $s\in\{\pm1\}$; this condition ensures that there are finitely many solutions, see Proposition~\ref{propFinChar}. For matrix updates, it utilizes the Sherman--Morrison formula so that there is no need to solve a system of equations in each step. However, the algorithm is exponential since it processes all orthants. To reduce the number of orthants to be processed,  various bounds of the overall solution set can be employed (see Section~\ref{sssSolvty}). The experiments by Hlad\'{\i}k~\cite{hladik2018bounds} showed that for random data using certain bounds, the orthant reduction is tremendous.

\emph{Others.}
Mansoori and Erfanian~\cite{mansoori2018dynamic} proposed a dynamic system model based on the projection function, enabling the AVE's analytic solution.

Yong at el.~\cite{yong2011feasible} and Achache and Hazzam in~\cite{achache2018solving} reformulated the AVE as a linear complementarity problem and then applied an interior point algorithm to solve it. Designing a suitable interior point algorithm for the original formulation seems not to have been investigated deeply so far.

The homotopy perturbation method is a method designed for solving nonlinear problems. This method gives the solution in an infinite series, converging to an accurate solution under some assumptions. Moosaei et al.~\cite{moosaei2015some} modified the homotopy perturbation method to solve the AVE. 

In \cite{yang2023modified}, a modified Barzilai--Borwein algorithm tailored for addressing the GAVE was introduced. The distinct advantage of this algorithm lies in its avoidance of solving subproblems and the unnecessary requirement of gradient information from the objective function during iterative sequences. The paper established the global convergence of the proposed algorithm under appropriate assumptions. 

\section{Further Aspects of the AVE} \label{Otherapro}

This section reviews other aspects of the AVE regarding solutions, solvability and relations to other areas.

\subsection{Minimum Norm Solutions for Absolute Value Equation}

If the solution set of the AVE is nonempty, it may have a finite number of solutions (at most $2^n$) or infinitely many solutions. 
In such cases, selecting a particular solution may be important, and a natural choice is a solution with the minimum norm; cf.~\cite{rosen1990minimum}. 

The minimum norm solution problem can be formulated as the optimization problem
\begin{align*}
    \min_{x\in\R^n}\ \|x\|^2 \ \ \st\ \ Ax-|x|=b.
\end{align*}
To solve it, Ketabchi and Moosaei~\cite{ketabchi2012minimum} transformed the problem to an unconstrained optimization problem with a once differentiable objective function; for this problem one, they proposed an extension of Newton’s method with the step size chosen by the Armijo rule. 
In contrast, Ketabchi and Moosaei~\cite{ketabchi2017augmented} investigated this model when 1-norm is used instead of the Euclidean norm. Despite the theoretical complexity, their numerical experiments demonstrated convergence to high accuracy in a few of iterations.

\subsection{Sparse Solution of Absolute Value Equation}

Finding a sparse solution to a system is a well-motivated problem. This NP-hard problem, which refers to an optimization problem involving the zero-norm in objectives or constraints, is a nonconvex and nonsmooth optimization
problem. Liu et al.~\cite{liu2016concave} focused on the problem of finding the sparse solution of the AVE, which could be expressed as follows:
 \begin{equation*}
\min\ \|x\|_0 \ \ \st\ \   Ax - |x| = b,
 \end{equation*}
where the zero-norm $\|x\|_0=$ is defined as the number of nonzero entries in~$x$; recall that it is not a real vector norm.
They proposed an algorithm based on a concave programming relaxation, whose main part was solving a series of linear programming problems.


\subsection{Optimal Correction of an Infeasible System of Absolute Value Equation}

Numerous reasons for the infeasibility of a system, including errors in data, errors in modeling, and many other reasons, can be argued. Because the remodeling of a problem, finding its errors, and generally removing its obstacles to feasibility might require a
considerable amount of time and expense and might result in yet another infeasible system, researchers are reluctant to do so. Researchers, therefore, focused on the optimal correction of the given system, where optimality means the least changes in data.

Ketabchi and Moosaei~\cite{ketabchi2012efficient} considered minimal correction in the right-hand side vector~$b$. This leads to the optimization problem
\begin{equation}\label{Eq6}
\min_{x\in {\R}^n}\ \| Ax-|x|-b\|^2.
\end{equation} 
Its optimal solution $x^*$ is the solution of the corrected AVE, and the corresponding corrected right-hand side $b'$ reads as $b'=Ax^*-|x^*|$.
To solve problem~\eqref{Eq6}, they adopted a generalized Newton method. 
 
Ketabchi et al.~\cite{ketabchi2013optimal} extended the result to the case of correction in both the coefficient matrix and the right-hand side vector. The problem is formulated as the optimization problem
\begin{align}\label{minOptCorrXRr}
 \min_{x,R,r}\ \|(R\mid r)\|_F
 \ \ \st\ \ (A+R)x-|x|=b+r,
\end{align}
where $\|\cdot\|_F$ denotes the Frobenius norm. 
It was shown that the problem could be reduced to the nonconvex fractional problem
\begin{align}\label{minOptCorrX}
 \min_{x\in \R^n}\ \frac{\| Ax-|x|-b\|^2}{1+\|x\|^2}.
\end{align}
If $x^*$ is its optimal solution, then the minimum correction values are
\begin{align*}
 R \coloneqq -\frac{ Ax^*-|x^*|-b}{1+\|x\|^2}(x^*)^\top,\quad
 r \coloneqq -\frac{ Ax^*-|x^*|-b}{1+\|x\|^2}.
\end{align*}
To solve the optimization problem \eqref{minOptCorrX}, Ketabchi et al.~\cite{ketabchi2013optimal} designed a genetic algorithm. Moosaei at el.~\cite{moosaei2021optimal} proposed several methods, including an exact but time-consuming orthant enumeration method, a variation of Newton's method, and a lower bound approximation based on the reformulation-linearization technique.

Since solving problem \eqref{minOptCorrX} sometimes leads to solutions with very large norms that are practically impossible to use, Moosaei at el.~\cite{moosaei2016tikhonov} utilized Tikhonov regularization
to control the norm of the resulting vector. 
Hashemi and Ketabchi~\cite{hashemi2019optimal} considered the infeasible GAVE and proposed an exact regularization method in the form
\begin{align*}
 \min_{x\in \R^n}\ \frac{\| Ax-B|x|-b\|^2}{1+\|x\|^2}+\lambda\|x\|^p_p,
\end{align*}    
where $\lambda>0$ is a parameter. 
They proposed a DC (difference of convex functions) algorithm to solve it. 
In the following study, Hashemi and Ketabchi~\cite{hashemi2020numerical} used the classical Dinkelbach's transform to handle the fractional objective function, and they also applied the smoothing functions to write it as a smooth problem.

Moosaei at el.~\cite{moosaei2021optimal} extended formula \eqref{minOptCorrX} to other norms. If we employ the spectral norm in \eqref{minOptCorrXRr}, then the formula remains the same. If we employ the Chebyshev (component-wise maximum) norm in \eqref{minOptCorrXRr}, then the formula takes the form
\begin{align*}
 \min_{x\in \R^n}\ \frac{\| Ax-|x|-b\|_{\infty}}{1+\|x\|_1}.
\end{align*}

Notice that the minimum in \eqref{minOptCorrX} may or may not be achieved. Moosaei at el.~\cite{moosaei2021optimal} derived some sufficient conditions for the existence of the minimum.


\subsection{Error Bounds}

In fact, most of the methods to solving the AVE compute only an approximate solution $x^*$ such that $Ax^*-|x^*|\approx b$. This is due to the essence of the iterative methods and because of the rounding errors of the floating-point arithmetic. A small residual value, i.e., $\|Ax^*-|x^*|-b\|<\varepsilon$, does not guarantee that $x^*$ is close to the true solution. That is why we need some error bounds.

Verification~\cite{Rum2010} is a technique that calculates numerically rigorous bounds for an approximate solution. To this end, methods of interval computations are often used. For the AVE, a verification algorithm was proposed by 
Wang et al.~\cite{wang2011interval}; it is based an an iterative $\epsilon$-inflation method to obtain an interval enclosure containing the unique solution. The width of this interval serves as an error estimate then.
Another verification techniques based on interval methods were proposed by Wang et al.~\cite{WanLiu2013,WanCao2017}. 
 
Classical-type error bounds were studied by Zamani and Hlad\'{\i}k~\cite{ZamHla2023a}. Let $[A-I_n,A+I_n]$ be regular and $x^*$ the unique solution of the AVE. Ten for every $x\in\R^n$ we have
\begin{align*}
    \|x-x^*\| \leq c(A) \cdot \|Ax-|x|-b\|,
\end{align*}
where 
\begin{align*}
   c(A) = \max_{|D|\leq I_n} \|(A-D)^{-1}\|
\end{align*}
is the proposed condition number for the AVE. A finite reduction is known: the value $c(A)$ is achieved for $D$ such that $|D|=I_n$. When matrix $p$-norm is used, the condition number is NP-hard to compute for $p\in\{1,\infty\}$, and the complexity is unknown for $p=2$. 

Besides this, also a relative condition number was introduced in the form
\begin{align*}
 c^*(A) = \max_{|D|\leq I_n} \|(A-D)^{-1}\| \cdot 
    \max_{|D|\leq I_n} \|A-D\|.
\end{align*}
If $b\not=0$, then the relative error bounds
\begin{align*}
c^*(A)^{-1} \frac{\|Ax-|x|-b\|}{\|b\|}
 \leq \frac{\|x-x^*\|}{\|x^*\|} 
 \leq c^*(A) \frac{\|Ax-|x|-b\|}{\|b\|}
\end{align*}
are valid.
Zamani and Hlad\'{\i}k~\cite{ZamHla2023a} also showed how to utilize the error bounds in proving convergence results for the generalized Newton and Picard methods.

\subsection{Relation to Interval Analysis}\label{ssRelIA}

There is a close connection between the AVE/GAVE and interval analysis. We already noticed the usage of interval matrices in expressing the conditions for the unique solvability of the AVE (Theorem~\ref{thmAveUnSol}) and the GAVE (Theorem~\ref{thmGaveCharT}). In this section, we look into other relations, particularly the relation of the AVE/GAVE with interval systems of linear equations.

\paragraph{GAVE in the Context of Interval Analysis.} 
Let $[\umace{A},\omace{A}]=[\Mid{A}-\Rad{A},\Mid{A}+\Rad{A}]$ be an interval matrix of size $n\times n$ and $[\uvr{b},\ovr{b}]=[\Mid{b}-\Rad{b},\Mid{b}+\Rad{b}]$ an interval vector of length~$n$ (see the notation on page~\pageref{dfIntMat}). Then, an interval system of linear equations~\cite{Neu1990} is the set of systems 
$$
Ax=b,\quad \mbox{where } A\in [\umace{A},\omace{A}],\ b\in[\uvr{b},\ovr{b}].
$$
In short we often write $[\umace{A},\omace{A}]x=[\uvr{b},\ovr{b}]$. The solution set is defined as
\begin{align*}
\Sigma=\{x\mmid Ax=b,\ \umace{A}\leq A\leq \omace{A},\ \uvr{b}\leq b\leq\ovr{b}\}.
\end{align*}
That is, $\Sigma$ is the set of all solutions of all realizations of the system of equations. The solution set has a complicated structure, similar to the AVE in a certain aspect. The set $\Sigma$ need not be convex, but in any orthant, it forms a convex polyhedron. Hence, $\Sigma$ may consist of a union of up to $2^n$ convex polyhedra. 

When $[\umace{A},\omace{A}]$ is regular, the solution set is connected. In this case, Rohn~\cite{Roh1989} discovered that the most important points of $\Sigma$, the vertices of its convex hull $\conv\Sigma$, are solutions of special GAVEs.

\begin{theorem}
If $[\umace{A},\omace{A}]$ is regular, then 
\begin{align*}
\conv\Sigma=\{x_s\mmid s\in\{\pm1\}^n\},
\end{align*}
where $x_s$ is the unique solution of the GAVE
\begin{align}\label{gaveThmRohnConv}
\Mid{A}x-\diag(s)\Rad{A}|x|=\Rad{b}+\diag(s)\Rad{b}.
\end{align}
\end{theorem}

Indeed, due to regularity of $[\umace{A},\omace{A}]$, the GAVE \eqref{gaveThmRohnConv} has a unique solution (see Theorem~\ref{thmgg}).
To solve the GAVE, Rohn also proposed a so-called \emph{sign accord algorithm}; see Algorithm~\ref{algSignancord} for a generalized version. 

GAVE systems also appear in different characterizations of the regularity of interval matrices. Rohn\cite{Roh2009} surveyed 40 necessary and sufficient conditions for regularity, and many of them have the form of the GAVE. We do not present all of them; as a teaser we chose just one: Interval matrix $[\umace{A},\omace{A}]$ is regular if and only if for each $s\in\{\pm1\}^n$, the GAVE
$$
\Mid{A}x-\diag(s)\Rad{A}|x|=s
$$
has a solution.

\paragraph{GAVE by Means of Interval Systems.} 
Absolute value systems appear often in the context of interval linear equations. The opposite direction, i.e., to express the solution set of the AVE or GAVE by means of an interval system, is less common and less studied but still possible.

Consider an interval system of linear inequalities $[\umace{A},\omace{A}]x\leq[\uvr{b},\ovr{b}]$, which is a shortage for the family
$$
Ax\leq b,\quad \mbox{where } A\in [\umace{A},\omace{A}],\ b\in[\uvr{b},\ovr{b}].
$$
There are two basic types of solutions associated to it~\cite{Roh2006a}. A vector $x$ is a \emph{weak} [\emph{strong}] solution if it satisfies the system $Ax\leq b$ for some [for every] $A\in [\umace{A},\omace{A}]$ and $b\in[\uvr{b},\ovr{b}]$. Weak solutions are characterized by the Gerlach theorem as \cite{Ger1981,Roh2006a}
\begin{align}\label{ineqGerlach}
 \Mid{A}x - \Rad{A}|x| \leq \ovr{b},    
\end{align}
whereas the strong solutions are described by the system
\begin{align}\label{ineqStrSol}
\Mid{A}x + \Rad{A}|x| \leq \uvr{b}.
\end{align}

Now, we rewrite the AVE equivalently as double inequalities
$$
Ax-|x|\leq b,\ \ -Ax+|x|\leq-b.
$$
By \eqref{ineqGerlach}, the former characterizes the weak solutions of the interval system $[A-I_n,A+I_n]x\leq b$.
By \eqref{ineqStrSol}, the latter inequality characterizes the strong solutions of the interval system $-[A-I_n,A+I_n]x\leq -b$.  To sum up, we have the following statement~\cite{HlaMos2022u}; we use $\exists$ and $\forall$ quantifiers to denote weak and strong solutions, respectively.

\begin{proposition}
The AVE system $Ax-|x|=b$ is equivalent to the interval system of linear inequalities
\begin{align*}
[A-I_n,A+I_n]^{\exists}x\leq b,\ \
[A-I_n,A+I_n]^{\forall}x\geq b.
\end{align*}
\end{proposition}

Notice that by swapping the quantifiers, the interval system
\begin{align*}
[A-I_n,A+I_n]^{\forall}x\leq b,\ \
[A-I_n,A+I_n]^{\exists}x\geq b.
\end{align*}
then characterizes a different AVE, namely $Ax+|x|=b$.

\section{Challenges and Open Problems} \label{Challenges}

Despite this considerable success and papers published, many challenges and critical questions remain unanswered. In this section, we attempt to highlight some of the main issues concerning the AVE and GAVE from different perspectives and outline some open problems and future research directions.

    
        

         
        



\paragraph{Complexity.} 
Since the AVE is an NP-hard problem, it would be desirable to identify more polynomially solvable sub-problems. The techniques from parameterized complexity can help here, too; the complexity could be parameterized by a suitable characteristic such as the rank of the constraint matrix.
          
\paragraph{Overdetermined and Underdetermined AVE.} 
In this paper, we considered merely the square AVE and GAVE, that is, absolute value systems of $n$ equations with $n$ variables. Overdetermined and underdetermined systems can appear, too. For instance, in the proof of Theorem~\ref{thmNpHardSP}, the reduction from the Set-Partitioning problem first resulted in an overdetermined GAVE system of $n+1$ equations with $n$ variables. To the best of our knowledge, nonsquare AVE and GAVE have been almost untouched so far.
 
\paragraph{Absolute Value Inequalities.} 
There is a natural step from absolute value equations to absolute value inequalities. While the inequality system $Ax+|x|\leq b$ describes a convex polyhedron a thus it is easy to deal with it in many aspects, the inequality system $Ax+|x|\geq b$ is again tough.

Absolute value inequalities often appear in the area of interval systems of linear equations and inequalities (see Section~\ref{ssRelIA}). Recall that the weak and strong solutions of an interval system of linear inequalities $[\umace{A},\omace{A}]x\leq[\uvr{b},\ovr{b}]$ are characterized by \eqref{ineqGerlach} and \eqref{ineqStrSol}, respectively. Analogously, the weak solutions of an interval system of linear equations $[\umace{A},\omace{A}]x=[\uvr{b},\ovr{b}]$ are described by an absolute value system \cite{OetPra1964,Roh2006a}
\begin{align*}
 \Mid{A}x - \Rad{A}|x| \leq \ovr{b},\ \ 
 \Mid{A}x + \Rad{A}|x| \geq \uvr{b},    
\end{align*}

Most of what is known about absolute value inequalities was derived in the context of interval linear systems. However, there are also few independent results~\cite{HlaHar2023au}, and surely there is room for deeper insights.

\paragraph{Absolute Value Programming.} 
Absolute value programming was introduced by Mangasarian~\cite{mangasarian2007absolute} and refers to mathematical programming problems involving absolute values. There have not been many results in this direction~\cite{HlaHar2023au}; the efforts have been expended primarily on absolute value equations.

There are, however, problems leading to linear programs with absolute values. We mention briefly a problem arising in interval linear programming. Consider a linear program
\begin{align*}
f(A,b,c)=\min\ c^{\top} x \ \ \st\ \ Ax\leq b,
\end{align*}
and suppose that matrix $A$ and vectors $b$ and $c$ come from an interval matrix $[\umace{A},\omace{A}]$ and interval vectors $[\uvr{b},\ovr{b}]$ and $[\uvr{c},\ovr{c}]$, respectively. The best and worst-case optimal values are defined, respectively, as
\begin{align*}
\unum{f}&\coloneqq \min\ f(A,b,c)\ \ \st\ \ 
 A\in[\umace{A},\omace{A}],\ 
  b\in[\uvr{b},\ovr{b}],\ c\in [\uvr{c},\ovr{c}],\\
\onum{f}&\coloneqq \max\ f(A,b,c)\ \ \st\ \ 
 A\in[\umace{A},\omace{A}],\ 
  b\in[\uvr{b},\ovr{b}],\ c\in [\uvr{c},\ovr{c}].
\end{align*}
These values can be expressed by means of an absolute value linear programs~\cite{Hla2012a}
\begin{align}\label{fUnumIlpAvp}
\unum{f} &=\min\ (\Mid{c})^{\top} x-(\Rad{c})^{\top} |x|\ \  \st\ \ 
  \Mid{A}x-\Rad{A}|x|\leq\ovr{b},\\
\label{fOnumIlpAvp}
\onum{f} &=\min\ (\Mid{c})^{\top} x+(\Rad{c})^{\top} |x|\ \  \st\ \ 
  \Mid{A}x+\Rad{A}|x|\leq\uvr{b}.
\end{align}
We cannot easily simplify \eqref{fUnumIlpAvp} since it was proved~\cite{GabMur2008} that computing $\unum{f}$ is an NP-hard problem. On the other hand, \eqref{fOnumIlpAvp} can be reformulated as an ordinary linear program
\begin{align*}
\onum{f} =\min\ (\Mid{c})^{\top} x+(\Rad{c})^{\top} y\ \  \st\ \ 
  \Mid{A}x+\Rad{A}y\leq\uvr{b},\ -y\leq x\leq y.
\end{align*}

\paragraph{Tensor Absolute Value Equations.}

In recent years, tensor absolute value equations have gained increasing attention. So far, researchers have considered various special tensor classes; they derived sufficient conditions for solvability and uniqueness and proposed numerical methods \cite{BeiKal2022, cui2022existence, CuiLia2022, DuZha2018, JiaLi2021}.
Nevertheless, there are many open problems and challenging directions. For example, a complete characterization of unique solvability is still unknown, even for those particular tensor absolute value equations having special structures.

\paragraph{Robust Solutions for the AVE and GAVE.}
In the realm of (generalized) absolute value equations, an open problem that warrants exploration is the development of robust solution methodologies that can effectively handle uncertain data. 

The challenge lies in extending the theory and algorithms for AVE and GAVE to accommodate uncertainty in the coefficients and constants of the equations. Robust solutions should provide stable results even when the data is subject to variations or perturbations within known uncertainty bounds.

In this regard, few results are known for absolute value equations. For systems with interval data, Raayatpanah et al.~\cite{RaaMoo2020} proposed a robust optimization approach to solve it, and \cite{HlaPta2023u} investigated its (unique) solvability and the overall solutions set. 
 Robust solution models for systems with data uncertainty in $\ell_1$ and $\ell_{\infty}$ norms were addressed in Lu et al.~\cite{LuMa2023}.

One key aspect to consider is the characterization and tight approximation of the set of feasible solutions for uncertain AVE and GAVE. Robust optimization techniques can be employed to identify solutions that remain valid under a range of uncertainty scenarios, ensuring stability and feasibility across different problem instances.

Additionally, the computational complexity of solving robust AVE and GAVE should be studied. Efficient algorithms that can handle uncertainty while maintaining acceptable computational performance are essential for practical applications.

Furthermore, developing robust solution methods for the AVE and GAVE that are amenable to scalable implementations is crucial. This would enable the application of these techniques to large-scale real-world problems, such as those encountered in engineering, finance, and machine learning.

\paragraph{Comparing Solution Methods for the AVE.}   
One of the significant challenges in studying absolute value equations lies in comparing the performance of existing solution methods. Conducting comprehensive comparisons regarding accuracy, computational time, complexity, and effectiveness under specific conditions on real data or randomly generated problems can provide valuable insights.
So far, no thorough numerical study and a comprehensive comparison of different methods has been carried out. 

Comparing the performance of various solution methods can help identify the strengths and weaknesses of each approach and provide guidelines for selecting the most suitable method for different scenarios. It allows researchers and practitioners to assess the trade-offs between accuracy and computational efficiency, enabling informed decision-making in practical applications.

Additionally, exploring the behavior and robustness of different solution methods under special conditions, such as ill-conditioned systems or highly sparse matrices, can reveal their limitations and highlight areas for improvement. These investigations can lead to the development of enhanced algorithms that offer superior performance in challenging problem instances.

\paragraph{Incorporating Sparsity in Complexity Bounds for AVE Algorithms.}
Practical optimization algorithms commonly leverage matrix sparsity for enhanced effectiveness. However, conventional complexity bounds often overlook the explicit consideration of matrix sparsity, particularly in the context of the AVE. The open problem addresses the need to develop a complexity bound for an algorithm that explicitly incorporates the sparsity of matrices in AVE scenarios.

This challenge is significant, as the existing connection between sparsity and the computational complexity of solving a single system of linear equations, especially in the context of AVE, is yet to be conclusively established \cite{pardalos1992open}. To tackle this open problem, future research could explore innovative methodologies to bridge the gap between sparsity-aware algorithms and complexity analyses. By doing so, researchers aim to establish a nuanced understanding of how sparsity and complexity interact in AVE scenarios. Addressing this open problem could lead to the development of more effective and tailored algorithms, ultimately enhancing the optimization of systems involving sparse matrices, particularly in the context of the AVE.

\section*{Conflict of Interest}
 The authors declare that they have no conflict of interest.

\bibliographystyle{spmpsci}      
\bibliography{mybib}   


\end{document}